\documentclass[ejs]{imsart}
\usepackage{thmtools}
\usepackage{amssymb}
\usepackage{amsmath}
\usepackage{amsthm}
\usepackage{xcolor}
\usepackage{hyperref}
\usepackage{enumitem}
\usepackage{bm}
\usepackage{bbm}
\usepackage{graphicx}
\usepackage{multirow}
\definecolor{carminepink}{rgb}{0.92, 0.3, 0.26}
\hypersetup{colorlinks=true,allcolors=carminepink} 

\newtheorem{theorem}{Theorem}

\newtheorem{lemma}{Lemma}
\newtheorem{corollary}{Corollary}

\newtheorem{remark}{Remark}

\newtheorem{condition}{Condition}
\newtheorem*{lemmafixed*}{Lemma 3.1 in \cite{dettedsphere}}
\DeclareSymbolFont{bbold}{U}{bbold}{m}{n}
\DeclareSymbolFontAlphabet{\mathbbold}{bbold}

\definecolor{bittersweet}{rgb}{1.0, 0.44, 0.37}

\newcommand*\diff{\mathop{}\!\mathrm{d}}
\newcommand\deq{\mathrel{\stackrel{\makebox[0pt]{\mbox{\normalfont\tiny d}}}{=}}}

\newcommand{\reals}{\mathbb{R}}					          
\newcommand{\integers}{\mathbb{Z}}
\newcommand{\Ex}{\mathbb{E}}				  
            
\newcommand{\Prob}{\mathbb{P}}

\newcommand{\alm}{a_{\ell, \mcoord}} 
\newcommand{\Y}{Y_{\ell, \mcoord}}    
\newcommand{\Yconj}{\bar{Y}_{\ell, \mcoord}}    
\newcommand{\Yconjprime}{\bar{Y}_{\ell^\prime, \mcoord^\prime}}

\newcommand{\indicator}[2]{\mathbbm{1}_{#1}\left(#2\right)}

\newcommand{\Mset}{\mathcal{M}_\ell}		
\newcommand{\Msetprime}{\mathcal{M}_{\ell^\prime}}

\newcommand{\mcoord}{\mathbf{m}}
\newcommand{\Sint}{\int_{\Sd}}

\newcommand{\primen}{\ell^{\prime}, \mcoord^{\prime}}
\newcommand{\ind}[1]{\mathbbold{1}{\bbra{#1}}}

\newcommand{\Ltwo}{{L^2\bra{\Sd}}}

\newcommand{\gegenbauer}[3]{C_{#1-#2}^{\left(#2+\frac{d-#3}{2}\right)}}
\newcommand{\sphcoord}{\left(\vartheta^{\left(1\right)},\ldots, \vartheta^{\left(d-1\right)},\varphi \right)}
\newcommand{\thetacord}[1]{\vartheta^{\left(#1\right)}}
\newcommand{\thetacoord}{\bm{\vartheta}}
\newcommand{\sphcoordshort}{\thetacoord,\varphi}

\definecolor{electricultramarine}{rgb}{0.25, 0.0, 1.0}

\newcommand{\qqqquad}{\quad \quad \quad \quad}

\newcommand{\alalm}{\tilde{a}_{\ell,\mcoord}}

\newcommand{\Ifunc}{I_{m_{j-1}, m_{j}}^{Q_{j-1}}\bra{m^\prime_{j-1}, m^\prime_{j}}} 
\newcommand{\Jfunc}{J_{m_{d-1}}^{Q_{d-1}}\bra{m_{d-1}^{\prime}}}
\newcommand{\Jfunczwei}{J_{m_{d-1}}^{2M}\bra{m_{d-1}^{\prime}}}

\newcommand{\peta}{	\eta\left(\ell, \mcoord;\ell +2s_0, \mcoord+2\mathbf{s}\right)}
\providecommand{\cc}[1]{\overline{#1}}
\providecommand{\abs}[1]{\left\vert#1\right\vert}

\newcommand{\revise}[1]{{#1}}
\providecommand{\bra}[1]{\left(#1\right)}			      
\providecommand{\sbra}[1]{\left[#1\right]}			    
\providecommand{\bbra}[1]{\left\{#1\right\}}			      

\newcommand{\Anorm}{h_{m_{j-1},m_{j};j}}
\newcommand{\Anormprime}{h_{m^\prime_{j-1},m^\prime_{j};j}}
\newcommand{\gegenbauergen}[2]{C_{#1}^{\left(#2\right)}}

\newcommand{\sumlm}{\sum_{\ell\geq 0}\sum_{\mcoord \in \Mset}} 	  
\newcommand{\sumshort}{\sum_{\ll,\mcoord}}

\newcommand{\Alias}{\tau\bra{\ell, \mcoord; \ell^\prime, \mcoord^\prime}}
\renewcommand{\ll}{\ell}

\newcommand{\Var}[1]{\operatorname{Var}\left(#1\right) }
 \newcommand{\Stwo}{\mathbb{S}^2}
  \newcommand{\Sd}{\mathbb{S}^d}
\newcommand{\spannung}{\operatorname{Span}}
\newcommand{\Cov}{\operatorname{Cov}}
\newcommand{\wpeso}[2]{w_{#1}^{\left(#2\right)}}

\begin{document}
\bibliographystyle{alpha}
\begin{frontmatter}
	\title{Aliasing effects for random fields over spheres of arbitrary dimension} 
	\runtitle{Aliasing effects for random fields over $\Sd$}
	
\begin{aug}
	\author{Claudio Durastanti \thanksref{t1}\ead[label=e1]{claudio.durastanti@gmail.com}}
	\address{Ruhr-Universit\"at  Bochum,  Faculty  of  Mathematics,  D-44780  Bochum,  Germany;\\ La Sapienza Università di Roma, Department of Basic and Applied Sciences for Engineering, 00161 Rome, Italy\\  
	\printead{e1}}
	
	\author{Tim Patschkowski \thanksref{t2}
	\ead[label=e2]{timpatschkowski@gmail.com}}
	\address{Ruhr-Universit\"at  Bochum,  Faculty  of  Mathematics,  D-44780 
		Bochum,  Germany\\
	\printead{e2}}
	
	\thankstext{t1}{C. Durastanti is partially supported by the Deutsche Forschungsgemeinschaft (\textit{GRK}  grant 2131: Ph\"anomene hoher Dimensionen in der Stochastik - Fluktuationen und Diskontinuit\"at); corresponding author.}
	
	\thankstext{t2}{T. Patschkowski is supported by the Deutsche Forschungsgemeinschaft (\textit{SFB} 823: Statistik nichtlinearer dynamischer Prozesse, Teilprojekt A1 and C1).}
	\runauthor{C. Durastanti \and T. Patschkowski}
\end{aug}

\begin{abstract}
In this paper, aliasing effects are investigated for random fields defined on the $d$-dimensional sphere $\Sd$ and reconstructed from discrete samples. First, we introduce the concept of an aliasing function on $\Sd$. The aliasing function \revise{allows one to} identify explicitly the aliases of a given harmonic coefficient in the Fourier decomposition. Then, we exploit this tool to establish the aliases of the harmonic coefficients approximated by means of the quadrature procedure named spherical uniform sampling. Subsequently, we study the consequences of the aliasing errors in the approximation of the angular power spectrum of an isotropic random field, the harmonic decomposition of its covariance function. Finally, we show that band-limited random fields are aliases-free, under the assumption of a sufficiently large amount of nodes in the quadrature rule. 
\end{abstract}



\tableofcontents

\end{frontmatter}


\section{Introduction}\label{sec:intro} 
\subsection{Overview}\label{sub:motiv}
We are concerned with the study of the aliasing effects for the harmonic expansion of a random field defined on the $d$-dimensional sphere $\Sd$. A spherical random field $T$ is a stochastic process defined over the unit sphere $\Sd$ and thus depending on the location $x=\left(\thetacoord,\varphi\right) = \sphcoord \in \Sd$, where \revise{$\thetacord{i} \in \left[0,\pi \right]$, for $i =1,\ldots,d-1$, and $\varphi \in \left[\left.0,2\pi\right)\right.$}. \revise{Harmonic analysis} has been proved to be an insightful tool to study several issues related to \revise{random} fields on the sphere and the development of spherical random fields in a series of spherical harmonics has many \revise{applications} in several branches of probability and statistics. We are referring, for example, to the study of the asymptotic behavior of the bispectrum of spherical random fields (see \cite{m}), their Euler-Poincar\'e characteristic  (see \cite{cammar2}), the estimation of their spectral parameters (\cite{dlm}), and the development of quantitative central limit theorems for nonlinear functional of corresponding random eigenfunctions (see \cite{MaRossi}).  Under some integrability conditions \revise{on $T$} (see Section \ref{sub:random}), the following harmonic expansion holds:
\begin{equation*}
	T \left(\sphcoordshort\right) = \sumshort \alm \Y \left(\sphcoordshort\right) ,
\end{equation*} 
where $\ell\in \mathbb{N}$ and $\mcoord = \left(m_1,\ldots,m_{d-1} \right)\in \mathbb{N}^{d-2}\otimes \mathbb{Z}$ are the harmonic (or wave) numbers.\\ 
The set of spherical harmonics $\Y=Y_{\ell,m_{1},\ldots,m_{d-1}}:\Sd \rightarrow \mathbb{C}$ provides an orthonormal basis for the space $\Ltwo=L^2\bra{\Sd,\diff x}$, where $\diff x$ is the uniform Lebesgue measure over $\Sd$ (see Section \ref{sub:harm}).
The harmonic coefficients $\alm = a_{\ll,m_1,\ldots,m_{d-1}}$ \revise{are given by}
\begin{equation}\label{eq:alma}
	\alm = \langle T, \Y \rangle_{\Ltwo} =\Sint T\bra{x}\cc{\Y}\bra{x} \diff x,
\end{equation}
\revise {and contain} all the stochastic information of $T\left(\sphcoordshort\right)$.\\
Nevertheless, the explicit computation of the integral \eqref{eq:alma} is an unachievable target in many experimental situations. Indeed, the measurements of $T\left(\sphcoordshort\right)$ can be in practical cases collected only over a finite sample of locations $$\left\{x_i=\bra{\thetacoord_i, \varphi_i}\in \Sd:i=1\ldots N\right\}.$$ As a consequence, for any choice of $\ell$ and $\mcoord$ the integral producing the harmonic coefficient $\alm$ is approximated by the sum of finitely many elements $T\left(x_i\right)$, $i=1\ldots,N$, the samples of the random field. \revise{As well-known in the literature, an exact reconstruction of the harmonic coefficients by means of finite sums represents a reachable target when considering band-limited random processes. Band-limited random processes are characterized by a bandwidth $L_0$, so that all the harmonic coefficients for $\ell \geq L_0$ are null. A suitable choice of a sampling theorem and the cardinality of the sampling points yields the exact reconstruction for the non-null coefficients (see also, for example, \cite{optdesranfield,stoer}). Further details will be discussed in Section \ref{sec:bandlimited}.\\ However, if the random field is not band-limited or if the sampling theorem is not properly selected, the approximation of the integral in \eqref{eq:alma} by a finite sum can produce the so-called} aliasing errors, that is, different coefficients become indistinguishable - aliases - of one another (see, for example, \cite{optdesranfield,stoer}). The set of coefficients, acting as aliases \revise{of each other}, depends specifically on the chosen sampling procedure.\\
The concept of aliasing (also known as confounding) comes from signal processing theory and related disciplines. In general, aliasing makes different signals \revise{indistinguishable} when sampled, and it can be produced when the reconstruction of the signal from samples is different from the original continuous one (see, for example, \cite[Chapter 1]{signal}). \\
The aliasing phenomenon arising in the harmonic expansion of a $2$-dimensional spherical random field has been investigated by \cite{LiNorth}. On the one hand, \revise{band-limited random fields} over $\mathbb{S}^2$, which can be roughly viewed as linear combinations of finitely many spherical harmonics, can be uniquely reconstructed with a sufficiently large sample size. On the other \revise{hand}, an explicit definition of the aliasing function, a crucial tool to identify the aliases of a given harmonic coefficient, is developed when the sampling is based on the combination of a Gauss-Legendre quadrature formula and a trapezoidal rule (see Section \ref{sec:aliasing} for further details). In many practical applications, this sampling procedure is the most convenient scheme to perform numerical analysis over the sphere (see, for example, \cite{atkinson,stoer,szego}). Further reasons of interest to study the aliasing effects in $\mathbb{S}^2$ have arisen in the field of optimal design of experiments. In \cite{dettemelas}, designs over $\mathbb{S}^2$ based on this sampling scheme have been proved to be optimal with respect to the whole set of Kiefer's $\Phi_p$-criteria, presented in \cite{kiefer}, that is, they are the most efficient among all the approximate designs for regression problems with spherical predictors. \\
Recently, interest has occurred in regression problems in spherical frameworks of arbitrary dimension and the related discretization problems (see, for example, \cite{langschwab}). In particular, in \cite{dettedsphere}, \revise{experimental designs} obtained by the discretization of the uniform distribution over $\Sd$ by means of the combination of the so-called Gegenbauer-Gauss quadrature rules (see Section \ref{sub:gegengauss} for further details) and a trapezoidal rule, have been proved to be optimal with respect not only to the aforementioned Kiefer's $\Phi_p$-criteria, but also to another class of orthogonally invariant information criteria, the $\Phi_{E_s}$-criteria. 
Given the \revise{increasing interest for spheres of dimension larger than $2$ (see Subsection \ref{sub:appl} for further details)}, it is therefore pivotal to carry out further investigations into the aliasing effects for random fields sampled over $\Sd$, $d >2$. On the one hand, this research improves the understanding of the behavior of the approximated harmonic coefficients when computed over discrete samplings, in particular over a spherical uniform sampling (see Section \ref{sub:unifo}). On the other hand, our investigations make \revise{extensive} use of the properties of the hyperspherical harmonics, \revise{thus providing} a deeper insight on their structure, carrying on with the results presented in \cite{dettedsphere}. \\ 

In this paper, we work under the following assumption: a spherical random field $T$ is observed over \revise{a} finite set of locations $\left\{x_i\in\Sd:i=1,\ldots,N\right\}$, the so-called sampling points, associated to the weights $\left\{w_i:i=1,\ldots,N\right\}$. Thus, for any set of harmonic numbers $\ell$ and $\mcoord$, the approximated - or aliased - harmonic coefficient is given by 
\begin{equation*}
	\alalm = \sum_{\ell^\prime, \mcoord^\prime} \Alias a_{\ell^\prime, \mcoord^\prime},
\end{equation*}
where \revise{$\Alias$ is the aforementioned aliasing function and is given by 
\begin{equation*}
\Alias=  \sum_{i=1}^{N} w_{i}  {Y}_{\primen}\bra{\thetacoord_i, \varphi_i} \Yconj \bra{\thetacoord_i, \varphi_i}\prod_{j=1}^{d-1} \left(\sin \thetacord{j}_i\right)^{d-j} 
.
\end{equation*} 
Further details can be found in Section \ref{sub:alfunc}}. The coefficient $a_{\ell^\prime,\mcoord^\prime}$ is said to be an alias of $\alm$ with intensity $\abs{\Alias}$ if $\Alias\neq 0$.\\
First, we study the general structure of the aliasing function under the very mild assumption that the sampling \revise{scheme} is separable with respect to the angular coordinates, that is, the sampling points $\left\{x_i:i=1,\ldots,N\right\}$ can be written as follows
\begin{equation*}
\left\{\left(\thetacord{1}_{k_0}, \ldots,\thetacord{d-1}_{k_{d-2}},\varphi_{k_{d-1}} \right): k_{j-1}=0,\ldots,Q_{j-1}-1 \quad \text{for } j=1, \ldots,d\right\},
\end{equation*}  
where $Q_0,Q_1,\ldots,Q_{d-1} \in  \mathbb{N}$ are defined so that $\prod_{j=0}^{d-1}Q_j=N$ (see Section \ref{sub:sep1}). \revise{Heuristically, a sample scheme is separable if a different discretization procedure is developed for each distinct coordinate.} 
Then, we investigate on the explicit structure of \revise{this} function and, consequently, on the identification of aliases assuming a spherical uniform design as the sampling procedure.\\
Second, under the assumption of isotropy, we consider the aliasing effects for the angular power spectrum of a random field, which describes the decomposition of the covariance function in terms of the frequency $\ell\geq 0$ (see Section \ref{sub:random}), providing information on the dependence structure of the random field.\\ 
\noindent Third, we investigate also on the aliasing effects for band-limited random fields. More specifically, we establish suitable conditions on the sample size in order to guarantee the annihilation of the aliasing phenomenon.

\subsection{Some applications and further research}\label{sub:appl}
\revise{An accurate characterization of the aliasing phenomena has great significance from both the points of view of theoretical statistics and its practical applications. More specifically, the analysis of spherical random fields over $\Sd$ is strongly motivated by a growing set of applications in several scientific disciplines, such as Cosmology and Astrophysics for $d=2$ (see, for example, \cite{bm, marpec2}), as well as in Medical Image Analysis (\cite{4Dmed,4dmed2}), Material Physics (\cite{MaSchu}), and Nuclear Physics (\cite{averyavery}) for $d>2$.}\\
\revise{As already mentioned, aliasing phenomena can be detected in all the experimental situations where harmonic coefficients are measured by means of a discretization of the integral given by Equation \eqref{eq:alma}. In this case, the presence of aliases can bring some crucial disadvantages for the experimenter.\\ 
In the classical optimal design approach (see for example \cite{dettedsphere}), the construction of experiments concerning spherical data is very sensitive to the aliasing effects. The outcomes of these experiments can be indeed affected by the aliasing of some terms belonging to the experimental design with other ones, potentially important but not included in the chosen model (see, for example, \cite{optimalaliasing}). According for instance to \cite{optdesranfield}, in the construction of experimental designs for the regression of random fields, the experimenter can exploit a first-order regression model, where interactions and aliasing are not considered. On the one hand, these designs are optimal to estimate primary effects. On the other hand, they can still present some undesirable aliasing effects producing some alias-depending bias. In this case, the information on the aliasing effects for each term is developed by means of the aforementioned aliasing function. The intensities of the aliases can be then collected in the so-called alias matrix (see, for further details, \cite{optimalaliasing}). The alias matrix depends specifically on the experimental design; for further details, the reader is referred, for example, to \cite{goosjones}. The construction of optimal designs minimizing the alias-depending bias subject to constraints on design efficiency, in the sense of the aforementioned optimality criteria (see, again, \cite{dettedsphere}), is therefore a topic of extreme interest (see also \cite{optimalaliasing}).}\\
\revise{Hyperspherical random fields in $\Sd$ can be furthermore exploited to study random fields defined over the unit ball $\mathbb{B}^{d-1}$ in $\mathbb{R}^{d-1}$, which currently represent a very challenging topic in data analysis. On the one hand, random fields defined over the unit ball $\mathbb{B}^3$ are a very useful tool, aimed to generate realistic three-dimensional models from observational data in several research branches of Cosmology, and other disciplines, such as, for instance, Medical Brain Imaging, and Seismology (see, respectively, \cite{dfhmp,brima,lm} and the references therein). %
On the other hand, the construction of a cubature formula on the unit ball is a complicated task. Indeed, even if it is theoretically known that the cubature points must correspond to the zeroes of Bessel functions of increasing degrees, in practice these points are not explicitly calculable (see, for example, \cite{lm}).  As proved in \cite{pexu}, under some mild smoothing conditions, this issue can be overcome by linking the construction of frames and the definition of cubature formulas on $\mathbb{B}^{d-1}$ with the ones on $\Sd$. More specifically, orthogonal polynomials on the unit sphere an those on unit ball can be
related by the following map
\begin{equation}\label{eqn:pexu}
x \in \mathbb{B}^{d-1} \mapsto x^\prime:=\left(x,\sqrt{1-\abs{x}^2}\right) \in \mathbb {S}^d,
\end{equation}
linking the points in $\mathbb{B}^{d-1}$ with the one in the upper hemisphere of $\Sd$ (see \cite{pexu}[Equation (4.5)]). We can thus define a distance $\mathbb{B}^{d-1}$
\begin{equation*}
d_{\mathbb{B}^{d-1}}(x, y) = \arccos \left(\langle x,y \rangle + \sqrt{1-\abs{x}}\sqrt{1-\abs{x}}\right), \quad x,y \in\mathbb{B}^{d-1},
\end{equation*} 
which corresponds to the geodesic distance on $\Sd$. The map given by \eqref{eqn:pexu} provides also a connection between (weighted) $L^p$-spaces on $\mathbb{B}^{d-1}$ and $L^p\left(\Sd\right)$.
This allows one to study random fields over the unit ball by means of objects defined over spheres of higher dimension. The understanding of aliasing effects over $\Sd$ becomes crucial to produce useful measurements related thus to these random fields.\\
By the point of view of applications,} in Medical Image Analysis the statistical representation of the shape of a brain region is commonly modeled as the realization of a Gaussian random field, defined across the entire surface of the region (see for example \cite{brima}). Many shape modeling frameworks in computational anatomy apply shape particularization techniques for cortical structures based on the spherical harmonic representation, to encode global shape features into a small number of coefficients (see \cite{4Dmed}). This data reduction technique, however, can not provide a proper representation with a single parametrization of multiple disconnected sub-cortical structures, specifically the left and right hippocampus and amygdala. The so-called 4D-hyperspherical harmonic representation of surface anatomy aims to solve this issue by means of a stereographic projection of an entire collection of disjoint 3-dimensional objects onto the hypersphere of dimension 4. Indeed, \revise{as aforementioned,} a stereographic projection embeds a 3-dimensional volume onto the surface of a 4-dimensional hypersphere, avoiding thus, the issues related to flatten 3-dimensional surfaces to the 3-dimensional sphere. Subsequently, any disconnected objects of dimension 3 can be projected onto a connected surface in $\mathbb{S}^4$, and, thus, represented as the linear combination of hyperspherical harmonics of dimension $4$  (see \cite{4dmed2}).\\
\revise{Finally, further investigations can be done to study the aliasing effects arising when alternative sampling schemes to the Gauss-Gegenbauer quadrature are taken into account. For example, we refer to the so-called equiangular sampling schemes, which involve a uniform discretization of all the angular coordinates, introduced by \cite{skuk}, and then developed, among others, by \cite{healydriscoll,mcewenwiaux}. Another relevant sampling scheme concerns the decomposition of the hypersphere into Voronoi cells (see, for example, \cite{npw1}). This sampling scheme allows one to build the so-called spherical needlets, a class of spherical wavelets featuring a wide range of applications in Statistics (see, for example, \cite{bkmpAoS, dlmejs, durastanti6}). In view of these applications, the aliasing effects related to this sampling procedure are of vibrant interest.}
\subsection{Organization of the paper}
This paper is structured as follows. In Section \ref{sec:prel}, we introduce some \revise{fundamental results} on the harmonic analysis over the $d$-dimensional sphere as well as a short review \revise{of} spherical random fields. Section \ref{sec:sampling} includes \revise{a short} overview on the so-called Gegenbauer-Gauss quadrature formula, crucial to build a spherical uniform sampling, and provides some auxiliary results. 
In Section \ref{sec:aliasing}, we present the main findings of this work. In particular, Theorem \ref{thm:auxiliary} describes the construction of the aliasing function $\Alias$ under the assumption of the separability of the sampling with respect to the angular components, while Theorem \ref{thm:main} identifies the aliases for any harmonic coefficient $\alm$ when the sampling is uniform. In Section \ref{sec:powerspec}, we study the aliasing effects for the angular power spectrum of an isotropic random field (see Theorem \ref{thm:power}), while in Section \ref{sec:bandlimited} we provide an algorithm to remove the aliasing effects for a band-limited random field sampled over a spherical uniform design (see Theorem \ref{thm:band}). \revise{Section \ref{sec:example} presents an explanatory example, while Section \ref{sec:proofs} collects all the proofs.}

\section{Preliminaries}\label{sec:prel}
This section collects some introductory results, concerning harmonic analysis and its application to spherical random fields. It also includes a quick overview on the Gegenbauer-Gauss formula. The reader is referred to \cite{steinweiss,atkinson,vilenkin} for further details about the harmonic analysis on the sphere, to \cite{randomfields} for a detailed description of random fields and their properties, while \cite{MaPeCUP} provides an extended description of spherical random fields over $\Stwo$. Further details concerning  the Gegenbauer-Gauss quadrature rule can be found in \cite{handbook,atkinson,stoer,szego}.\\

\subsection{Harmonic analysis on the sphere}\label{sub:harm} Let $\vartheta^{\bra{i}} \in \left[0,\pi\right]$, for $i=1,\ldots,d-1$, and $\varphi \in \left[ \left. 0, 2 \pi \right) \right.$ be the spherical polar coordinates over $\Sd$. \revise{From now} on, we will denote by $x=\bra{\thetacoord,\varphi}=\bra{\thetacord{1},\ldots,\thetacord{d-1},\varphi}$ the generic spherical coordinate, that is, the direction of a point on $\Sd$. Let the function $f:\left[0,\pi \right]^{d-1}\rightarrow \left[-1,1\right]$ be defined by
\begin{equation}\label{eq:ffunc}
f\left(\thetacoord\right)=f\left(\thetacord{1},\ldots,\thetacord{d-1}\right)= \prod_{j=1}^{d-1} \left(\sin \thetacord{j}\right)^{d-j}. 
\end{equation}
Thus, the uniform Lebesgue measure $\diff x$ over $\Sd$, namely, the element of the solid angle, is defined by
\begin{align*}\diff x = & \left(\sin \thetacord{1}\right)^{d-1} \diff\thetacord{1}  \left(\sin \thetacord{2}\right)^{d-2} \diff \thetacord{2} \ldots \sin \thetacord{d-1} \diff \thetacord{d-1}   \diff \varphi\\
=& f\left(\thetacord{1},\ldots,\thetacord{d-1}\right) \diff \thetacord{1} \ldots \diff \thetacord{d-1} \diff\varphi,
\end{align*}
such that the surface area of the hypersphere corresponds to
\begin{equation*}
\int_{\Sd} \diff x = \frac{2\pi^{\frac{d+1}{2}}} {\Gamma\left(\frac{d+1}{2}\right)},
\end{equation*}
\revise{where $\Gamma$ denotes the Gamma function.}
Let us denote by $\mathcal{H}_\ell$ the restriction of the space of harmonic homogeneous polynomials of order $\ell$ to $\Sd$. As well-known in the literature (see, for example, \cite{atkinson,steinweiss}), the space of square-integrable functions over $\Sd$ can be described as the direct sum of the spaces $\mathcal{H}_\ell$, that is,
\begin{equation*}
	\Ltwo = \underset{\ell \geq 0}{\bigoplus} \mathcal{H}_\ell.
\end{equation*}
For any integer $\ell \geq0$, \revise{from now} on called frequency, we define the following set 
\begin{align}\notag
	\Mset = & \left\{\mcoord
		\in \mathbb{Z}^{d-1}: m_1=0, \ldots, \ell; m_2 = 0, \ldots,m_{1};\ldots;m_{d-2}= 0, \ldots,m_{d-3};\right. \\ 
		&\left. m_{d-1} = -m_{d-2}, \ldots,m_{d-2}\right\}.\label{eq:Mset}
\end{align}
Following \cite{averywen,atkinson,vilenkin}, for any $\ell \geq 0$, it holds that
\begin{equation*}
	\mathcal{H}_\ell = \spannung\!\bbra{Y_{\ell,\mcoord}:\mcoord\in \Mset},
\end{equation*}
where, for $x \in \Sd$, $\Y=Y_{\ell,m_1,\ldots,m_{d-1}}:\Sd \rightarrow \mathbb{C}$ denotes the so-called spherical - or hyperspherical - harmonic of degree $\ell$ and order $\mcoord$. In other words, fixed $\ell \geq 0$, $\Mset$ appoints the finitely many vectors $\mcoord$ which identify the spherical harmonics spanning the space $\mathcal{H}_\ell$. \\
Another common approach to introduce spherical harmonics exploits the so-called $d$-spherical Laplace-Beltrami operator $\Delta_{\Sd}$ (see, for example, \cite{MaPeCUP}).
Fixed $\ell \geq 0$, the spherical harmonics $\Y\left(x\right)$ corresponding to any $m \in \Mset$ are the eigenfunctions of $\Delta_{\Sd}$ with eigenvalue \revise{$-\varepsilon_{\ell;d}$}, where $\varepsilon_{\ell;d}=\ell\bra{\ell+d-1}$, that is,
\begin{equation*}
	\left( \Delta_{\Sd}+\varepsilon_{\ell;d} \right) \Y\left(x\right) =0, \text{ for } x \in \Sd. 
\end{equation*}
As proved for example in \cite{averywen}, for any $\ell \geq 0$, the size of $\left\{\Y:\mcoord \in \Mset\right\}$, namely, the multiplicity of the set of spherical harmonics with eigenvalue $\varepsilon_{\ell;d}$, is given by
\begin{equation}\label{eqn:sizespher}
\Xi_d\left(\ell\right) = \frac{\left(2\ell+d-1\right)\left(\ell + d-2\right)!}{\ell!\left(d-1\right)!}.
\end{equation}
The set $\bbra{\Y\left(x\right):\ell \geq 0; \mcoord\in \Mset}$ provides therefore an orthonormal basis for $\Ltwo$. For any $g \in \Ltwo$, the following Fourier - or harmonic - expansion holds
\begin{equation*}
	g\bra{x}=\sumlm \alm \Y\bra{x}, \text{ for }x \in \Sd,
\end{equation*} 
where $\left\{\alm:\ell \geq 0; \mcoord \in \Mset\right\}$ are the so-called harmonic coefficients, given by the integral
\begin{equation*}
	\alm = \langle g , \Y \rangle_{\Ltwo} = \Sint g\bra{x}\Yconj\bra{x} \diff x.
\end{equation*}
\revise{From now} on, for the sake of notational simplicity, we fix $m_{0}=\ell$. Furthermore, we will use indifferently the two equivalent short and long notations $\Y\left(x\right)$ and $Y_{\ll,m_{1},\ldots,m_{d-1}}\sphcoord$. Following \cite{averywen}, the hyperspherical harmonics are defined by
\begin{equation}\label{eqn:spharm}
	\Y\left(x\right)= \frac{1}{\sqrt{2\pi}} \prod_{j=1}^{d-1} \left( \Anorm 
	\gegenbauer{m_{j-1}}{m_{j}}{j} \left(\cos \thetacord{j}\right)\left(\sin \thetacord{j}\right)^{m_{j}} \right)e^{im_{d-1}\varphi},
\end{equation}
where $h_{m_{k-1},m_{k};k}$ is a normalizing constant, given by
\begin{equation}\label{eq:normalizing}
	\Anorm  = \left(\frac{2^{2m_{j}+d-j-2}\left(m_{j-1}-m_{j}\right)!\left(2m_{j-1} +d-j\right)\Gamma^2\left(m_{j}+\frac{d-j}{2}\right) }{\pi \left(m_{j-1}+m_{j}+d-j-1\right)!}\right)^{\frac{1}{2}}.
\end{equation} 
The function $\gegenbauergen{n}{\alpha}:\left[-1,1\right]\rightarrow \reals$, $\alpha \in\left[\left.-1/2,\infty\right)\right.$, is the Gegenbauer (or ultraspherical) polynomial of degree 
$n$ and parameter $\alpha$. Following for example \cite{handbook,szego}, they are orthogonal with respect to the measure 
\begin{equation*}
	\nu_{\alpha}\left(t\right) = \left(1-t^2\right)^{\alpha-\frac{1}{2}} \indicator{\left[-1,1\right]}{t},
\end{equation*}
that is,
\begin{equation}\label{eq:gegenortho}
	\int_{-1}^{1} \gegenbauergen{n}{\alpha} \left(t\right) \gegenbauergen{n^\prime}{\alpha} \left(t\right) \nu_{\alpha}\left(t\right) \diff t = \frac{\pi 2^{1-2\alpha} \Gamma\left(n+2\alpha\right)}{n!\left(n+\alpha\right)\Gamma^2\left(\alpha\right)} \delta_{n}^{n^\prime}, 
\end{equation}
see, for example, \cite[Formula 4.7.15]{szego}. \revise{Note that for $\alpha = 1/2$ the Gegenbauer polynomials reduce to the Legendre polynomials, while for $\alpha = 0,1$ the Gegenbauer polynomials reduce to the Chebyshev polynomials of the first and of the second kind respectively (see \cite[Chapter 22, Section 5]{handbook}).}
\\Roughly speaking, each hyperspherical harmonic in \eqref{eqn:spharm} can be viewed as product of a complex exponential function and a set of Gegenbauer polynomials, whose orders and parameters are properly nested and normalized to guarantee orthonormality, that is,
\begin{equation*}
	\int_{\Sd} \Y\left(x\right)  \Yconjprime\left(x\right) \diff x = \delta_{\ell}^{\ell^\prime}\prod_{k=1}^{d-1}\delta_{m_{k}}^{m_{k}^\prime}.
\end{equation*} 
 Hyperspherical harmonics feature also the following property, known as addition formula (see, for example, \cite{averywen}):
\begin{align}\notag
	\sum_{\mcoord \in \Mset}  \Y\bra{x} \Yconjprime\bra{x^\prime} & = \frac{\left(2\ell+d-1\right)\Gamma\left(\frac{d+1}{2}\right)\left(\ell+d-2\right)!}{2 \pi^{\frac{d+1}{2}} \left(d-1\right)! \ell! } C_{\ell}^{\left(\frac{d-1}{2}\right)}\left(\langle x,x^\prime\rangle\right)\\ & =:K_\ell \bra{x,x^\prime},\label{eqn:addition}
\end{align}
where $\langle \cdot,\cdot\rangle$ is the standard inner product in $L^2\bra{\reals^{d+1}}$. 
Note that $K_\ell$ can be viewed as the kernel of the projector over the harmonic space $\mathcal{H}_\ell$, the restriction to the sphere of the space of homogeneous and harmonic polynomials of order $\ell$. The projection $\mathcal{P}_\ell$ of $g \in \Ltwo$ onto $\mathcal{H}_\ell$ is given by
\begin{equation*}
	\mathcal{P}_\ell \sbra{g} \bra{x} = \Sint g\bra{y} K_{\ell} \bra{x,y} \diff y, \quad x\in \Sd.
\end{equation*}
It follows that
\begin{equation*}
	\mathcal{P}_\ell \sbra{g} \bra{x}=\sum_{\mcoord \in \Mset} \alm\Y\left(x\right), \quad \text{ for } x\in\Sd,
\end{equation*}
and that any function $g \in \Ltwo$ can be rewritten as the sum of projections over the spaces $\mathcal{H}_\ell$,
\begin{equation*}
	g\left(x\right) = \sum_{\ell\geq 0} \mathcal{P}_\ell \sbra{g} \bra{x} , \text{ for }x \in \Sd. 
\end{equation*}
\subsection{Spherical random fields} \label{sub:random}
Given a probability space $\left\{\Omega, \mathcal{F},\Prob \right\}$, a spherical random field $T_{\omega}\left(x\right)$, $\omega \in \Omega$ and $x \in \Sd$, describes a stochastic process defined the sphere $\Sd$. \revise{From now} on, the dependence on $\omega \in \Omega$ will be omitted and the random field will be denoted by $T\left(x\right)$, $x \in \Sd$, for the sake of the simplicity (see also \cite{randomfields}).\\
If $T$ has a finite second moment, that is, $\Ex\sbra{\abs{T\left(x\right)}^2} <\infty$ for all $x \in \Sd$, a spherical random field  can be decomposed in terms of the projections over the space $\mathcal{H}_{\ell}$, $\ell \geq 0$, so that 
\begin{equation}\label{eqn:Tproj}
T\left(x\right)  = \sum_{\ell \geq 0} T_{\ell}\left(x\right), \quad x \in \Sd,
\end{equation}
where $T_{\ell}\left(x\right) = \mathcal{P}_{\ell}\sbra{T}\left(x\right)$. 
Each projector onto $\mathcal{H}_\ell$ can be described as a linear combination of finitely many hyperspherical harmonics, 
\begin{equation}\label{eqn:Fourier_expansion}
T_\ell\left(x\right)   = \sum_{\mcoord \in \Mset} \alm \Y \left(x\right), \quad x \in \Sd.
\end{equation}
As in the deterministic case described in Section \ref{sub:harm},  for any $\ell \geq 0$ and $\mcoord \in \Mset$, the random harmonic coefficient is defined by
\begin{equation}\label{eqn:fieldcoeff}
\alm = \Sint T\left(x\right) \Yconj \left(x\right) \diff x.
\end{equation}
The random harmonic coefficients contain all the stochastic information of the random field $T$, namely,  $\alm=\alm \left(\omega\right)$, for $\omega \in \Omega$, $\ell \geq 0$ and $\mcoord \in \Mset$.\\
A random field is said to be band-limited if there exists a bandwidth $L_0\in \mathbb{N}$, so that $\alm = 0$ for any $\ell > L_0$, whenever $m \in \Mset$. In this case, it holds that
	\begin{equation}\label{eq:randomfield}
	T\left(x\right)= \sum_{\ell=0}^{L_0}\sum_{\mcoord\in\Mset} \alm \Y\left(x\right), \quad x \in \Sd.
	\end{equation}
By the practical point of view, band-limited random fields provide a useful  approximation of fields with harmonic coefficients decaying fast enough as the frequency $\ell$ grows.\\
Let us define the expectation $\mu\left(x\right)=\Ex\sbra{T\left(x\right)}$; the covariance function $\Upsilon: \Sd\times\Sd \rightarrow \reals$ of the random field $T$ is given by
\begin{equation}\label{eq:gamma}
\revise{\Upsilon\bra{x,x^\prime}} =\Ex\sbra{\bra{T\left(x\right)-\mu\left(x\right)}\bra{\bar{T}\left(x^\prime\right)-\bar{\mu}\left(x^\prime\right)}},
\end{equation} 
where, for $z \in \mathbb{C}$, $\bar{z}$ denotes its complex conjugate.
\noindent Without losing any generality, assume that $T$ is centered, so that, for $x,x^\prime \in \Sd$, it holds that
\begin{align*}
& \mu\left(x\right)=0 \\
& \revise{\Upsilon\bra{x,x^\prime}} =\Ex\sbra{T\left(x\right)\bar{T}\left(x^\prime\right)}. 
\end{align*}
Let $\gamma:\Sd\times\Sd \rightarrow \left[0,\pi\right],\gamma\left(x,x^\prime\right)=\arccos \langle x,x^\prime\rangle_{\reals^{d+1}} $ 
be the geodesic distance between $x,x^\prime \in \Sd$. A spherical random field is said to be isotropic if it is invariant in distribution with respect to rotations of the coordinate system or, more precisely, 
\begin{equation*}
	T\left(x\right) \deq T\left(Rx\right), \text{ for } x \in \Sd, R \in SO\left(d+1\right),
\end{equation*}
where $\deq$ denotes equality in distribution, and $SO\left(d+1\right)$ is the so-called special group of rotations in $\reals^{d+1}$. Following \cite{bkmpBer,bm,MaPeCUP}, if the random field is isotropic, then $\Upsilon$ depends only on $\gamma$ and its variance $\sigma^2\left(x\right)=\Upsilon\bra{x,x}$ does not depend on the location $x \in \Sd$, so that it holds that
\begin{equation*}
\sigma^2 \left(x\right)=\Ex\sbra{\abs{T\left(x\right)}^2} = \sigma^2, \quad \text{ for all }x \in \Sd,
\end{equation*}
where $\sigma^2 \in \reals^+$.
The covariance function itself can be therefore rewritten in terms of its dependence on the distance between $x$ and $x^\prime$, so that 
\begin{equation*}
\revise{\Upsilon\bra{x,x^\prime}=\Upsilon \bra{\gamma\left(x,x^\prime\right)}}.
\end{equation*}
Let us finally define the correlation function $\rho:\left[-1,1\right]\rightarrow\left[-1,1\right]$, which is invariant with respect to rotations when the random field is isotropic, that is
\begin{equation}\label{eq:isotropy}
\revise{\rho\bra{\cos \gamma\left(x,x^\prime\right)} = \frac{\Upsilon\bra{x,x^\prime}}{\sqrt{\Upsilon\bra{x,x}\Upsilon\bra{x^\prime,x^\prime}}}=\frac{\Upsilon \bra{\gamma\left(x,x^\prime\right)}}{\sigma^2}, \quad x,x^\prime = \Sd}
\end{equation}
As far as the random harmonic coefficients $\{\alm: \ell \geq 0, \mcoord \in \Mset\}$ are concerned, since $\mu\left(x\right)=0$ for $x \in \Sd$, we have that $\Ex\sbra{\alm}= 0$. \revise{On the one hand, the Fourier expansion of $T$ can be read as a decomposition of the field into a sequence of uncorrelated random variables, preserving its spectral characteristics, that is, \begin{equation}\label{eq:covariance}
\Cov\bra{\alm,a_{\ell^\prime,\mcoord^\prime}}=\Ex\sbra{\alm \bar{a}_{\ell^\prime,\mcoord^\prime}} = C_\ell \delta_{\ell}^{\ell^\prime}\prod_{k=1}^{d-1}\delta_{m_{k}}^{m_{k}^\prime},
\end{equation}
where $\{C_\ell:\ll\geq 0\}$ is the so-called angular power spectrum of $T$.\\ 
On the other hand, the spectral decomposition of the covariance function is given by
\begin{equation*}
	\Upsilon\left(x,x'\right) = \sum_{\ell\geq 0} C_{\ell} K_\ell \left(x,x^\prime\right),
	\end{equation*}
where we rewrite the covariance function in terms of the projection kernel corresponding to the frequency level $\ell$.
Combining \eqref{eqn:addition}, \eqref{eq:gamma} and \eqref{eq:covariance}, the angular power spectrum of a random field can be viewed as the harmonic decomposition of its covariance function and can be rewritten as the average
\begin{equation}\label{eq;clsum}
C_\ell=\frac{1}{\Xi_d\left(\ell\right)} \sum_{\mcoord \in \Mset} \Var{\alm},
\end{equation}
where $\Xi_d\left(\ell\right)$ is given by \eqref{eqn:sizespher}
(see, for example, \cite{m} for $d=2$).}
\section{The Gauss-Gegenbauer quadrature formula and the spherical uniform design}\label{sec:sampling} This section includes a quick overview on the Gegenbauer-Gauss formula. We also introduce the spherical uniform sampling and two related auxiliary results. Further details concerning the Gegenbauer-Gauss quadrature rule can be found in \cite{handbook,atkinson,stoer,szego}, while the spherical uniform sampling is presented by \cite{dettedsphere}.\\

\subsection{Separability of the sampling}\label{sub:sep1} We first introduce a very mild condition on the sampling procedure. Generalizing the proposal introduced by \cite{LiNorth} on $\Stwo$ to $\Sd$, $d>2$, here we consider a discretization scheme produced by  the combination of $d$ one-dimensional quadrature rules, with respect to the coordinates $\thetacord{j}$, $j=1,\ldots, d-1$, and  $\varphi$. \\ 
More specifically, we introduce the following condition on the sampling points and weights. 
\begin{condition}[Separability of the sampling scheme]\label{cond:separability}
Fix $Q_0,Q_1,\ldots,Q_{d-1} \in  \mathbb{N}$, so that $N =\prod_{j=0}^{d-1}Q_j$. For any $j=1,\ldots,d$, there exists a finite sequence of positive real-valued weights 
\begin{equation}\label{eq:weight}
\left\{\wpeso{k_{j-1}}{j}:k_{j-1}=0,\ldots,Q_{j-1}-1\right\},
\end{equation}
so that 
\begin{equation*}
\sum_{k_{j-1}=0}^{Q_{j-1}-1}\wpeso{k_{j-1}}{j}=1.
\end{equation*}
The sampling points $\left\{x_i:i=1,\ldots,N\right\}$ are component-wise given by
\begin{equation}\label{eq:samplingangles}
\left\{\left(\thetacord{1}_{k_0}, \ldots,\thetacord{d-1}_{k_{d-2}},\varphi_{k_{d-1}} \right): k_{j-1}=0,\ldots,Q_{j-1}-1 \quad \text{for } j=1, \ldots,d\right\}.
\end{equation}
\end{condition}
Roughly speaking, each sequence in \eqref{eq:weight} corresponds to the set of weights for a quadrature formula with respect to the $j$-th angular component of the angle vector $x=\left(\thetacord{1},\ldots,\thetacord{d-1}, \varphi\right)$. The subscript index is related to the harmonic numbers $\ell=m_0, m_1,\ldots,m_{d-1}$. \\
Each value of the index $i^\ast\in \left\{1,\ldots,N\right\}$ corresponds uniquely to a suitable choice of values $\left\{k_0^\ast,\ldots,k_{d-1}^\ast\right\}$, while the related weight $w_{i^\ast}$ is given by  
\begin{equation*}
w_{i^\ast} = \prod_{j=1}^{d} \wpeso{k_{j-1}^\ast}{j}.
\end{equation*}
\subsection{The Gauss-Gegenbauer quadrature formula}\label{sub:gegengauss}
In general, a quadrature rule denotes an approximation of a definite integral of a function by means of a weighted sum of function values, estimated at specified points within the domain of integration (see, for example, \cite{stoer}). In particular, a $r$-point Gaussian quadrature rule is a
formula specifically built to yield an exact result for polynomials of degree smaller or equal to $2r - 1$, after a suitable choice of 
the points and weights $\left\{t_k,\omega_k: k=0,\ldots, r-1\right\}$. For this reason, it is also called quadrature formula of degree $2r-1$. The domain of integration is conventionally taken as $\left[-1,
1\right]$, and the choice of points and weights usually depends on the so-called weight function $a$, whereas the integral can be written in the form $\int_{-1}^{1} p\left(t\right)a\left(t\right)\diff t$. Here $p\left(t\right)$ is approximately polynomial, and $a\left(t\right)\in L^1\left(\left[-1,1\right]\right)$ is a well-known function. In this case, a proper selection of $\left\{t_k,\omega_k: k=0,\ldots, r-1\right\}$ yields
\begin{equation*}
\int_{-1}^{1} p\left(t\right)a\left(t\right)\diff t = \sum_{k=0}^{r-1} \omega_k p\left(t_k\right).
\end{equation*}
\revise{From now on, while the letter $\omega$ will concern weights related to quadrature formulas for coordinates on the interval $\left[-1,1\right]$, the letter $w$ will denote weights related to quadrature formulas for angular coordinates.}\\
Following for example \cite{stoer}, it can be shown that the quadrature points can be chosen as the roots of some polynomial belonging to some suitable class of orthogonal polynomials, depending on the function $a$. \\When $a\left(t\right)=1$ for all $t \in \left[-1,1\right]$, the associated polynomials are the Legendre polynomials. In this case, the method is then known as Gauss-Legendre quadrature (see \cite[Formula 25.4.29]{handbook}). Such a method is widely used in the $2$-dimensional spherical framework (see, for example, \cite{atkinson}), and the aliases produced by this formula were largely investigated in \cite{LiNorth}). \\
 More in general, as stated in \cite[Formula 25.4.33]{handbook}, when $a\left(t\right)=a_{\alpha,\beta}\left(t\right)=\left(1-t\right)^\alpha\left(1+t\right)^\beta$, the method is known as the Gauss-Jacobi quadrature formula, since it makes use of the Jacobi polynomials 
(see also \cite[p.47]{szego}). Since it is well-known that Jacobi polynomials reduce to Gegenbauer polynomials when $\alpha=\beta$ (see, for example, \cite[Formula 4.1.5]{szego}), we refer to the quadrature rule denoted by a weight function $\nu_{\alpha}\left(t\right)$ (equal to $a_{\alpha,\beta}\left(t\right)$ for $\alpha=\beta$) as the Gauss-Gegenbauer quadrature (see, for example, \cite{gegecub}).\\
Subsequently, the discrete uniform sampling over the sphere is obtained by combining a trapezoidal rule for the angle $\varphi$ and $\left(d-1\right)$ Gauss-Gegenbauer quadrature rules for the coordinates $\thetacord{j}$, for $j=1,\ldots,d-1$, with weight function $a_j\left(t\right)= \nu_{\alpha\left(j\right)}\left(t\right)$, $\alpha\left(j\right)=d-1-j$. \\
This method has been described in details by \cite[Lemma 3.1]{dettedsphere} in the framework of optimal design for regression problems with spherical predictors. Indeed, by the theoretical point of view, the (continuous) uniform distribution on the sphere provides an optimal design for experiments on the unit sphere, but this distribution is not implementable as a design in real experiments (for more details, see \cite[Theorem 3.1]{dettedsphere}). Thus, a set of equivalent discrete designs is established by means of the combination of the following quadrature formulas over the sphere, written as in \cite[Lemma 3.1]{dettedsphere}, to which we refer to for a proof.
\revise{\begin{lemma}[Gauss-Gegenbauer quadrature]\label{def:quadrature}
	Let $a\in L^1\left(\left[-1,1\right]\right)$ be a positive weight function so that $\bar{a}=\int_{-1}^{1} a\left(t\right)\diff t$. Consider also the set of $r\in \mathbb{N}$ points $-1\leq t_0 < \ldots < t_{r-1} \leq 1$ , associated to the positive weights $\omega_0,\ldots,\omega_{r-1}$ such that $\sum_{k=0}^{r-1}\omega_k =1$. Then the set of points and weights$\left\{t_k,\omega_k:k=0,\ldots,r-1\right\}$ generates a quadrature formula of degree $z \geq r$, namely,
	\begin{equation}\label{eq:quadrat}
	\int_{-1}^{1} a\left(t\right) t^p \diff t = \bar{a} \sum_{k=0}^{r-1} \omega_k t_{k}^{p}, \quad \text{for } p=0, \ldots, z, 
	\end{equation} 
	if and only if the following conditions are satisfied:
	\begin{enumerate}
		\item The polynomial $
		\prod_{k=0}^{r-1}  \left(t-t_k\right)$ is orthogonal to all polynomials of degree smaller or equal to $z-r$ with respect to $a\left(t\right)$,
		\begin{equation*}
		\int_{-1}^{1} \prod_{k=0}^{r-1}  \left(t-t_k\right) a\left(t\right)t^p \diff t = 0, \quad \text{for } p=0,\ldots, z-r;
		\end{equation*}
		\item the weights $\omega_k$ are given by 
		\begin{equation}\label{eq:gegweight}
		\omega_k = \frac{1}{\bar{a}}\int_{-1}^{1} a\left(t\right) \lambda_k\left(t\right) \diff t, \quad\text{for } k=0,\ldots,r-1, 
		\end{equation}
		where $\lambda_k\left(t\right)$ is the $k$-th Lagrange interpolation formula with nodes $t_0,\ldots,t_{r-1}$, given by
		\begin{equation*}
		\lambda_k\left(t\right) = \prod _{i=0, i\neq k}^{r-1} \frac{t-t_i}{t_i-t_k}.
		\end{equation*} 
	\end{enumerate}
\end{lemma}}
\subsection{The spherical uniform sampling}\label{sub:unifo} Assume now  $z=2Q_0$ in Definition \ref{def:quadrature}. Following \cite[Formula 4.7.15]{szego} (see also \eqref{eq:gegenortho}), the Gegenbauer polynomials $\gegenbauergen{n}{\alpha}$ are orthogonal with respect to $a\left(t\right)=\nu_{\alpha}\left(t\right)$. Fixed $n$, the real-valued $n$ roots of $\gegenbauergen{n}{\alpha}$ have multiplicity $1$ and are located in the interval $\left[-1,1\right]$. Thus, it follows that for any $r \in \left\{Q_0+1, \ldots, 2Q_0 \right\}$, there exists at least one set of points and weights $\left\{t^{\left(j\right)}_k, \omega^{\left(j\right)}_k:k=0,\ldots,r-1\right\}$, $j=1\ldots,d-1$, generating a quadrature formula \eqref{eq:quadrat} with $a\left(t\right)=a_{j}\left(t\right)= \nu_{\alpha\left(j\right)}\left(t\right)$, and $\alpha\left(j\right)=d-1-j$. \\
\revise{In Lemma \ref{def:quadrature} above, we have recalled a set of quadrature formulas for the interval $\left[-1,1\right]$, each of those associated to the corresponding weight function $\nu_{\alpha\left(j\right)}$, for $j=1,\ldots,d-1$, The following Condition exploits properly these quadrature formulas for $\thetacoord$, combined with a trapezoidal rule for $\varphi$, to establish a well-defined uniform distribution over the sphere of arbitrary dimension $d$ (see also, for example, \cite{atkinson,dettedsphere}). Observe that this choice yields a suitable quadrature formula for each angular component in $\Sd$.}
\begin{condition}[Spherical uniform sampling]\label{cond:uni}
	Assume that Condition \ref{cond:separability} holds and fix $M\in \mathbb{N}$ so that $Q_{d-1}=2M$. The sampling with respect to $\varphi$ is uniform, so that 
	for any $k_{d-1}=0,\ldots, 2M-1$, it holds that
	\begin{align}
	\label{eq:azimuthpoint}\varphi_{k_{d-1}}=& \frac{k_{d-1}\pi}{M};\\
	\label{eq:azimuthweight}	\wpeso{k_{d-1}}{d}= & \frac{\pi}{M}.
	\end{align} 
	The sampling with respect to each component $\thetacord{j}$, $j=1,\ldots,d-1$ has the form 
	\begin{align}
	\label{eq:thetasamp}	\thetacord{j}_{k_{j-1}}=&\arccos\left(t_{k_{j-1}}^{\left(j\right)}\right);\\
	\label{eq:thetawei} \wpeso{k_{j-1}}{j}=&\frac{\omega_{k_{j-1}}^{\left(j\right)}}{\left(\sin \thetacord{j}_{k_{j-1}} \right) ^{d-j}},
	\end{align}
	where, for any $j=1,\ldots,d-1$, $\left\{t_{k_{j-1}}:k_{j-1}=0,\ldots,Q_{j-1}-1 \right\}$ in \eqref{eq:thetasamp} are the zeros of $\gegenbauergen{Q_{j-1}}{\frac{d-j}{2}}$, while  $\left\{\omega_{k_{j-1}}:k_{j-1}=0,\ldots,Q_{j-1}-1 \right\}$ in \eqref{eq:thetawei} are the corresponding weights in the Gauss-Gegenbauer framework, given by \eqref{eq:gegweight} in Definition \eqref{def:quadrature}.
\end{condition} 
\revise{As already discussed in \cite{atkinson,dettedsphere}, the Gauss-Gegenbauer quadrature in Lemma \ref{def:quadrature} is characterized by a unitary sum of the weights for each component, while Condition \ref{cond:uni} guarantees orthonormality for spherical harmonics $Y_{\ell,\mcoord}$ and $Y_{\ell^\prime,\mcoord^\prime}$ so that $\ell+\ell^\prime \leq 2 Q_0$, that is,
\begin{align*}
&\sum_{k_0=0}^{Q_0-1}\ldots \sum_{k_{d-1}=0}^{Q_{d-1}-1} \left(\prod_{j=1}^{d}w_{k_{j-1}^{\left(j\right)}}\right)  Y_{\ell,\mcoord}\left(\thetacoord_{k_0,\ldots,k_{d-2}},\varphi_{k_{d-1}}\right)\cc{Y}_{\ell^\prime,\mcoord^\prime}\left(\thetacoord_{k_0,\ldots,k_{d-2}},\varphi_{k_{d-1}}\right)\\ &\qqqquad\qqqquad\qqqquad \qqqquad\qqqquad\qqqquad= \delta_{\ell}^{\ell^\prime}\prod_{k=1}^{d-1}\delta_{m_k}^{m_k^\prime}, \end{align*}
for $\ell+\ell^\prime \leq 2 Q_0$.}\\
We present now two auxiliary results crucial to prove Theorem \ref{thm:main}, referring to the aliasing effects under Condition \ref{cond:uni}. Their proofs can be found in Section \ref{sub:proofaux} \\
The first Lemma establishes the parity properties of the cubature points and weights for each angular component $\thetacord{j}$ with respect to $\thetacord{j}=\pi/2$, for $j=1,\ldots,d-1$. Indeed, due to the parity formula $\gegenbauergen{r}{\alpha} \left(-t\right)= \left(-1\right)^r \gegenbauergen{r}{\alpha} \left(t\right)$ (see \cite[Formula 4.7.4]{szego}), the roots of $\gegenbauergen{r}{\alpha}\left(t\right)$, $t_1, \ldots, t_r$, are symmetric with respect to $0$, namely, $t_k = -t_{r-k-1}$ for $k=0,\ldots,\left[r/2\right]$. As a consequence, the following lemma holds.
\begin{lemma}\label{lemma:sampling}
	Let the cubature points and weights be given by \eqref{eq:thetasamp} and \eqref{eq:thetawei} respectively in the framework described by Definition \ref{def:quadrature}. Hence, for any $j=1,\ldots,d-1$,  it holds that 
	\begin{align*}
	&	\vartheta_{k_{j-1}}^{\left(j\right)} = \pi - 	\vartheta_{Q_{j-1}- k_{j-1}-1}^{\left(j\right)};\\
	&	\wpeso{k_{j-1}}{j} = \wpeso{Q_{j-1}-k_{j-1}-1}{j}.
	\end{align*}
\end{lemma}
The next result exploits Lemma \ref{lemma:sampling} to develop parity properties on the Gauss-Gegenbauer quadrature formula.
\begin{lemma}\label{lemma:simmetry}
	Let $\psi\in \left[0, \pi \right]$, and $j=1,\ldots,d-1$.  Let $m_i \in \mcoord$, with $m_0=\ell$ and $m_i^\prime \in \mcoord^\prime$, with $m_0^\prime = \ell^\prime$ and define, for  $j=1,\ldots,d-1$, 
	\begin{equation*}
	G_j\left(\psi\right) = \gegenbauer{m_{j-1}}{m_{j}}{j}\left(\cos\psi\right) \gegenbauer{m^\prime_{j-1}}{m^\prime_{j}}{j}\left(\cos \psi\right) \left(\sin\psi\right) ^{d-j}.
	\end{equation*}
	Then it holds that
	\begin{equation}\label{eq:parity2}
	G_j\left(\pi - \psi\right)  = \bra{-1}^{m_{j-1}+m_{j-1}^\prime-m_{j}-m_{j}^\prime }G\left(\psi\right).
	\end{equation}
	Furthermore, for $Q \in \mathbb{N}$, let $\left\{\psi_k:k=0,\ldots, Q-1 \right\}$ and $\left\{w_k:k=0,\ldots, Q-1 \right\}$ be samples of points and weights in $\left[-1,1\right]$ so that  for  $k=0,\ldots,\left[Q/2\right]$
	\begin{align*}
	&\psi_k= \psi_{Q-1-k}, \\
	& w_k= w_{Q-1-k}, 
	\end{align*}
	where $\left[\cdot\right]$, $t \in \reals$ denotes the floor function. Then, if $\bra{m_{j-1}+m_{j-1}^\prime-m_{j}-m_{j}^\prime}= 2c+1$, $c\in \mathbb{N}$, it holds that
	\begin{equation}\label{eq:parity3}
	\sum_{k=0}^{Q-1} w_k G_j\left(\psi_k\right)=0.
	\end{equation}	 
\end{lemma}

\section{Aliasing effects on the sphere}\label{sec:aliasing}
This section presents our main results concerning the aliasing phenomenon for $d$-dimensional spherical random fields. 
First, we define the aliasing function, the key tool to explicitly  determine the aliases for any given harmonic coefficient. Then, we study the aliasing function \revise{and the set} of harmonic numbers identifying the aliases for any given coefficient $\alm$ in two different cases. The proof of the theorems presented in this section are collected in Section \ref{sub:main}.\\
As a first step, we just assume that the aliasing function is separable with respect to the angular components. This assumption is very mild, as it reflects both the separability of the spherical harmonics and the practical convenience of choosing separable sampling points, with respect to the angular coordinates.\\
As a second step, we study the aliasing effects under the assumption that the sample comes from a spherical uniform design.\\

\subsection{The aliasing function}\label{sub:alfunc}
\noindent In practical applications, the measurements of the random fields can be sampled  only over a finite number of locations on $\Sd$. As a straightforward consequence, the integral \eqref{eqn:fieldcoeff} can not be explicitly computed, but it has to be replaced by a sum of finitely many samples of $T$. \\
Fixed a sample size $N \in \mathbb{N}$ and given a set of sampling points over $\Sd$ $$\bbra{x_i=\left(\thetacoord_i, \varphi_i\right)\in \Sd:i=1,\ldots,N},$$ the measurements of the spherical random field $T$ are collected in the sample $\bbra{T\bra{x_i}:i=1,\ldots,N}$. For any $\ell\geq0$ and $\mcoord \in \Mset$, the approximated harmonic coefficient is given by 
\begin{equation}\label{eq:firstalias}
\alalm =  \sum_{i=1}^{N} w_{i} T\bra{\thetacoord_i,\varphi_i} \Yconj \bra{\thetacoord_i,\varphi_i} f\left(\thetacoord_i\right),
\end{equation}
where $f\left(\thetacoord\right)$ is given by \eqref{eq:ffunc}. Combining \eqref{eqn:Tproj} and \eqref{eqn:Fourier_expansion} with \eqref{eq:firstalias}  yields
\begin{align}
\notag \alalm = & \sum_{i=1}^{N} w_{i} \bra{\sum_{\ell^\prime \geq 0}\sum_{\mcoord^\prime\in \Msetprime} a_{\primen}{Y}_{\primen}\bra{\thetacoord_i, \varphi_i}} \Yconj\bra{\thetacoord_i, \varphi_i} f\left(\vartheta_i\right) \\
=&\sum_{\ell^\prime \geq 0}\sum_{\mcoord^\prime\in \Msetprime}\Alias a_{\primen}.
\label{eqn:aliaseries} 
\end{align}
where $\Alias$ is given by
\begin{equation}
\Alias=  \sum_{i=1}^{N} w_{i}  {Y}_{\primen}\bra{\thetacoord_i, \varphi_i} \Yconj \bra{\thetacoord_i, \varphi_i} f\left(\thetacoord_i \right) 
.\label{eqn:aliasfunct}
\end{equation}
\revise{From now} on, we will refer to $\Alias$ as the \textit{aliasing function} and to $\alalm$ as the \revise{\textit{random aliased coefficient}}. 
For $\ell^\prime \neq \ell$ and $m^\prime \neq m$, the coefficients $a_{\primen}$ in \eqref{eqn:aliaseries} are called \textit{aliases} of $\alm$ if $\Alias \neq 0$. \revise{Note that if the random field $T$ is centered, it follows that 
\begin{align*}
\Ex\left[\alalm\right] =&\sum_{\ell^\prime \geq 0}\sum_{\mcoord^\prime\in \Msetprime}\Alias \Ex\left[a_{\primen}\right] = 0.
\end{align*} }
As stated by \cite{LiNorth} for the case $d=2$, on the one hand, the following equality
\begin{equation*}
\Alias=\delta^{\ell}_{\ell^\prime}\prod_{i=1}^{d-1}\delta^{m_{i}}_{m_{i}^\prime},
\end{equation*}
is a necessary and sufficient condition to identify $\alm$ and $\alalm$. This equality does not hold in general (see Section \ref{sec:bandlimited}). On the other hand, fixed  $\ell, \ell^\prime, \mcoord$ and $\mcoord^\prime$, 
 if $\Alias \neq 0$, that is, $a_{\primen}$ is an alias of $\alm$, its intensity, denoting how large is the contribution of this alias, is given by $\abs{\Alias}$. \\
 The total amount of aliases in \eqref{eqn:aliaseries} and the corresponding intensity depends specifically on the choice of the sampling points $\bbra{x_i: i=1,\ldots, N}$ over $\Sd$, which characterizes entirely the subsequent structure of \eqref{eqn:aliasfunct}. In other words, every setting chosen for the sampling points leads to a specific set of aliases, described by the corresponding aliasing function.\\ 
 Here we study the aliasing function $\Alias$ first in a more general framework, under the assumption of a separable sampling with respect to the angular coordinates in Section \ref{sub:ali1}, and then for a discrete version of the spherical uniform distribution in Section \ref{sub:ali2}. 
 \subsection{The separability of the aliasing function}\label{sub:ali1} Let us assume now that the assumptions of Condition \ref{cond:separability} hold. Thus,
given $Q_0,Q_1,\ldots,Q_{d-1} \in  \mathbb{N}$, so that $N =\prod_{j=0}^{d-1}Q_j$, for $j=1,\ldots,d-1$, the corresponding set of quadrature points and weights is given by
\begin{equation*}
\left\{\left(\vartheta_{{k}_{j-1}}^{\left(j\right)},\wpeso{k_{j-1}}{j}\right)\in\left[0,\pi\right]\times[0,1]:k_{j-1}=0,\ldots,Q_{j-1}-1\right\},
\end{equation*}
while, for $j=d$, we have that
\begin{equation*}
\left\{\left(\varphi_{{k}_{d-1}},\wpeso{k_{d-1}}{d}\right)\in\left[0,2\pi\right]\times[0,1]:k_{d-1}=0,\ldots,Q_{d-1}-1\right\},
\end{equation*}
As a straightforward consequence, \revise{we obtain the following result.}
\begin{theorem}\label{thm:auxiliary}
	\revise{Under Condition \ref{cond:separability},} it holds that
\begin{equation}\label{eq:aliasdec}
\Alias=\frac{1}{2\pi}\prod_{j=1}^{d-1} \Anorm \Anormprime   \Ifunc \Jfunc, 
\end{equation}
where $\Anorm$ is given by \eqref{eq:normalizing} and
\begin{align}
 & \Jfunc = \sum_{k_{d-1}=0}^{Q_{d-1}-1} w_{k_{d-1}}^{\left(d\right)} e^{i\bra{m_{d-1}^{\prime}-m_{d-1}}\varphi_{k_{d-1}}};\label{eq:Jfunc} \\
  & \Ifunc\! =\!\!\! \sum_{k_{j-1}=0}^{Q_{j-1}-1} \!\!\!w_{k_{j-1}}^{\left(j\right)} \left( \sin \vartheta_{k_{j-1}}^{\left(j\right)} \right)^{m_{j}\!+m_{j}^{\prime}\!+d-j}\notag\\
  &\qqqquad\qqqquad\quad \,	\cdot\gegenbauer{m_{j-1}}{m_{j}}{j} \left(\!\cos \thetacord{j}_{k_{j-1}}\!\right) \gegenbauer{m^{\prime}_{j-1}}{m^{\prime}_{j}}{j} \left(\!\cos \thetacord{j}_{k_{j-1}}\!\right).\label{eq:Ifunc} 
\end{align}
\end{theorem}
\begin{remark} Loosely speaking, the function $\Alias$ can be rewritten as a chain of products of functions, pairwise coupled by two indexes $m_{j},m_j^\prime$, $j=1,\dots,d-2$. Indeed, as shown by \eqref{eqn:spharm}, each angular component $\thetacord{j}$ is related to two harmonic numbers $m_{j-1}$ and $m_j$. While $\Jfunc$ is concerned with the discretization of components along the azimuthal angle $\varphi$, the factors $\Ifunc$, $j=1,\ldots,d-1$, represent the discretization along the $j$-th component of the vector $\thetacoord$. Finally, the multiplicative factor $\Anorm$ comes from the normalization of hyperspherical harmonics in \eqref{eqn:spharm}. \\
\revise{From now} on, we will refer to $\Ifunc$, for $j=1,\ldots,d-1$, and $\Jfunc$ as the aliasing (function) $j$-th and $d$-th factors respectively.
\end{remark}

\subsection{Aliasing and spherical uniform designs}\label{sub:ali2} As already mentioned in Section \ref{sub:motiv}, the motivations behind the study of this particular setting come from two different sources. On the one hand, the uniform design is largely used in the framework on numerical analysis over the sphere (see \cite{atkinson,stoer,szego}). On the other hand, in the field of mathematical statistics, the spherical uniform sampling has be proved to be the the most efficient design with respect to a large set of optimality criteria such as the Kiefer's $\Phi_p$- as well as the $\Phi_{E_s}$-criteria, in the framework of optimal designs of experiments (see \cite{dettedsphere}). Furthermore, in Remark \ref{rem:comp}, we show that our findings align with the results established  \cite{LiNorth}) for the two-dimensional case. The example described in Section \ref{sec:example} establishes explicitly the set of aliases of a given harmonic coefficient.\\
The main results of this section, stated in the forthcoming Theorem \ref{thm:main}, require some further notation, produced in Remark \ref{rem:preliminary}. 
\begin{remark}\label{rem:preliminary} Let us fix preliminarily $m_0=\ell$. \revise{From now} on, $\mathbf{s}=\left(s_1,\ldots,s_{d-1}\right)\in \mathbb{Z}^{d-1}$ will denote a $\left(d-1\right)$-vector of indices, while $\mathbf{Q}=\left(Q_0,Q_1,\ldots,Q_{d-1}\right)$ is a $d$-vector collecting the cardinality of the quadrature nodes for each angular component in $\left(\thetacoord,\varphi\right)$. Following Lemmas \ref{lemma:sampling} and \ref{lemma:simmetry}, for $\ell\geq 0$ and $\mcoord\in \Mset$, Theorem \ref{thm:main} establishes that the aliases for $\alm$ are identified by the harmonic numbers $\left(\ell^\prime,\mcoord^\prime\right)$, so that $\abs{m_j-m_j^\prime}=2s_j$, $j=0,\ldots,d-1$. The aliases of $\alm$ take thus the form
	\begin{equation*}
	a_{\ell+2s_0, \mcoord+2\mathbf{s}} = a_{\ell+2s_0, m_1+2s_1,\ldots,m_{d-2}+2s_{d-2},m_{d-1}+2rM}, 
	\end{equation*} 
	where the indices $s_0,\ldots,s_{d-1}$ belong to suitable sets defined as follows. For the index $s_0$, we define 
	\begin{equation}\label{eq:Dset}
	D_0=D_0\left(\ell\right)  = \left\{ s_0 \in \mathbb{Z}: s_0 \geq  -\frac{\ell}{2}\right\}. 
	\end{equation}
	Then, for $j=1, \ldots, d-2$, we have that
	\begin{equation}\label{eq:Hset}
	H^{\left(j\right)}_{m_j}\left(m_{j-1}+2s_{j-1}\right) = \left\{s_{j}\in \mathbb{Z}:-\frac{m_j}{2}\leq s_{j}\leq \frac{ \left(m_{j-1}+ 2 s_{j-1}\right)-m_{j} }{2}\right\}.
	\end{equation}
	Finally, the last index $s_{d-1}$, characterizing the trapezoidal rule on $\varphi$, depends on the constant $M$ given in Condition \ref{cond:uni}, so that $s_{d-1} = r M$, where $r$ belongs to the following set,
	\begin{align}\notag
	&R^M_{m_{d-1}} \bra{m_{d-2}+2s_{d-2}}\\ &\quad:=\bbra{r\in \integers :-\frac{\left(m_{d-2}+2s_{d-2}\right) +m_{d-1}}{2M}\leq r \leq \frac{\left(m_{d-2}+2s_{d-2}\right)-m_{d-1}}{2M}}.\label{eq:Rfunc}
	\end{align}
	Notice that for $j=1, \ldots, d-1$ each index $s_j$,  belongs to a set whose size depends on the value of $s_{j-1}$. Furthermore, while $D_0\left(\ell\right)$ provides just a lower bound for $s_0$, each $H^{\left(j\right)}_{m_j}\left(m_{j-1}+2s_{j-1}\right)$, $j=1,\ldots,d-1$, features only finitely many elements.\\ 
	Let us now define the following sets, 	
	\begin{align}
	&A_{0} =A_{0}  \left(\ell,Q_0\right)=\left\{s_{0}\in  \integers: -\frac{\ell}{2}\leq s_0\leq Q_0 - \ell -1\right\}; \label{eq:A0}\\
	&B_{0} =B_{0} \left( \ell,Q_0\right)= \left\{s_{0}\in  \integers: Q_0 - \ell \leq s_0\leq \infty \right\},\label{eq:B0}
	\end{align}
	and, for $j=1,\ldots,d-2$,
	\begin{align}
	&A_{j}=A_{j} \left(m_{j},Q_j\right)=\left\{s_{j}\in \integers: -\frac{m_j}{2}\leq s_j\leq Q_j - m_{j} -1\right\};\label{eq:Aj}\\
	&B_{j}=B_{j} \left( m_{j-1},m_{j},s_{j-1} ,Q_j\right)= \left\{s_{j}\in \integers: Q_j - m_{j} \leq s_{j} \leq \frac{m_{j-1}-m_{j}}{2}+s_{j-1}\right\}.\label{eq:Bj}
	\end{align}
	Observe that the definition of $A_j$ and $B_j$ is formally correct only if $Q_j - m_{j} <\frac{m_{j-1}-m_{j}}{2}+s_{j-1}$, that is, $s_{j-1}>Q_{j}-\frac{m_{j-1}+m_{j}}{2}$. Thus, \revise{from now} on, for $s_{j-1}\leq Q_{j}-\frac{m_{j-1}+m_{j}}{2}$, we consider 
	\begin{align}
	&A_{j}=\left\{s_{j}\in \integers: -\frac{m_j}{2}\leq s_j \leq  \frac{m_{j-1}-m_{j}}{2}+s_{j-1}\right\};\label{eq:Ajbis}\\
	&B_{j}=\emptyset,\label{eq:Bjbis}
	\end{align}
	to take into account all the possible combinations of $s_{j-1}$ and $Q_j$. 
	It is straightforward to observe that 
	\begin{equation*}
	D_0 = A_0\cup B_0, \quad H^{\left(j\right)}_{m_j}\left(m_{j-1}+2s_{j-1}\right)=A_j\cup B_j, \quad \text{for } j=1,\ldots,d-2.
	\end{equation*} 
	Define now the following sets  
	\begin{align}
	\label{eq:H0} 	&H^{\left(j\right);0}_{m_j}\left(m_{j-1}+2s_{j-1}\right) =H^{\left(j\right)}_{m_j}\left(m_{j-1}+2s_{j-1}\right) \cap \left\{s_j \neq 0 \right\};\\
	\label{eq:R0}& R^{M;0}_{m_{d-1}} \bra{m_{d-2}+2s_{d-2}}\cap\left\{r \neq 0 \right\},
	\end{align}
	which are equal to $H^{\left(j\right)}_{m_{j-1},m_j}\left(s_{j-1}\right) $ and $R^{M}_{m_{d-1}} \bra{m_{d-2}+2s_{d-2}}$ respectively, but omitting the null value. Finally, we define, for $j=1,\ldots,d-2$,
	\begin{align}
	\Delta_{j}=&\Delta_{j} \left(m_{j-1}+2s_{j-1},m_{j},Q_{j-1},s_{j-1}\right) \notag\\ 
	=& \left\{s_j \in \integers: s_j \in \left( H^{\left(j\right);0}_{m_{j}}\left(m_{j-1}+2s_{j-1}\right) \ind{s_{j-1} \in A_{j-1}}\right.\right.\notag\\ &\left. \left.+
	H^{\left(j\right)}_{m_{j}}\left(m_{j-1}+2s_{j-1}\right) \ind{s_{j-1} \in B_{j-1}}\right)\right\}.\label{eq:Dother}
	\end{align} 
	while 
	\begin{align}\notag
	\Delta_{d-1} = & 	\Delta_{d-1}\left(m_{d-2}+2s_{d-2},m_{d-1},M,s_{d-2}\right)\\
	=&  \left\{s_{d-1}= Mr; M=Q_{d-1}/2, r \in \integers: r\! \in \!\left( \notag R^{M,0}_{m_{d-1}}\left(m_{d-2}+2s_{d-2}\right)\right. \right.\\ &\left.  \left. \cdot\ind{s_{d-2} \in A_{d-2}}+  R^{M}_{m_{d-1}}\left(m_{d-2}+2s_{d-2}\right)\ind{s_{d-2} \in B_{d-2}}\right)  \right\},\label{eq:Dlast}
	\end{align}
	In other words, when $s_j \in \Delta_j$, it can take any value in $H^{\left(j\right)}_{m_{j-1}}\left(m_{j-1}+2s_{j-1}\right)$ if $s_{j-1}\in B_{j-1}$. Otherwise, if $s_{j-1}\in A_{j-1}$, it can take any value in the set $H^{\left(j\right)}_{m_{j-1}}\left(m_{j-1}+2s_{j-1}\right)$ except to the null value.\\
	We collect these sets together with the notation
	\begin{equation}\label{eqn:zeta}
	Z_{\ell, \mcoord}^{\mathbf{Q}}=\left\{\left(s_1, \ldots, s_{d-1} \right) :s_1 \in \Delta_1, \ldots, s_{d-1} \in \Delta_{d-1};s_1 \geq \ldots \geq s_{d-1}\right\}.
	\end{equation}
	Finally, we define
	\begin{align}\notag
	&\eta\left(\ell, \mcoord;\ell +2s_0, \mcoord+2\mathbf{s}\right)\\
	& \qquad
	=\prod_{j=1}^{d-1} \Anorm  h_{m_{j-1}+2s_{j-1},m_j+2s_j;j}  I_{m_{j-1},m_j}^{Q_{j-1}}\left(m_{j-1}+2s_{j-1},m_{j}+2s_{j} \right),\label{eq:Aliasnew}
	\end{align} 
	where $\Anorm$ and  $I_{m_{j-1},m_j}^{Q_{j-1}}\left(m_{j-1}+2s_{j-1},m_{j}+2s_{j}\right)$ are defined by \eqref{eq:normalizing} and \eqref{eq:Ifunc} respectively, and corresponding to $\Alias$ as given by \eqref{eq:aliasdec}, with $\ell^\prime=\ell+2s_0$, $\mcoord^\prime = \mcoord+2\mathbf{s}$ and $\Jfunc=2\pi$.
\end{remark}
\begin{theorem}\label{thm:main}
	Assuming that Condition \ref{cond:uni} holds, for any $\ell \geq 0$ and $\mcoord \in \Mset$, the aliased harmonic coefficient defined in \eqref{eqn:aliaseries} is given by
	\begin{equation}
	\alalm =   \alm 
	+ \sum_{s_0 \in D_0\left(\ell\right)} \sum_{\mathbf{s}\in Z_{\ell, \mcoord}^{\mathbf{Q}}} \peta a_{\ell+2s_0, \mcoord +2\mathbf{s}}, \label{eq:thmmain}
	\end{equation}
	where $\peta$ is defined by \eqref{eq:Aliasnew}, while the sets $D_0\left(\ell\right)$ and $Z_{\ell, \mcoord}^{\mathbf{Q}}$ are given by \eqref{eq:Dset} and \eqref{eqn:zeta}. 
\end{theorem}
\begin{remark}[Comparison with the $2$-dimensional case] \label{rem:comp}
The aliasing effects over $\Stwo$ have been studied by \cite{LiNorth}, involving a trapezoidal rule for the coordinate $\vartheta$ and the Gauss-Laplace quadrature formula for the angle $\vartheta$. More formally, fixed $Q\in \mathbb{N}$, a quadrature formula is obtained by a set of $Q$ points and weights $\left\{\theta_k, w_k:k=0,\ldots, Q-1\right\}$, obtained as in Definition \ref{def:quadrature}. The points $\left\{\theta_k:k=0,\ldots, Q-1\right\}$ are, in this case, the nodes of the Legendre polynomial of order $Q$. Recall that, for $d=2$, $m$ does not identify a vector of harmonic numbers, but just an integer, defined so that $-\ell \leq m\leq \ell$. Thus, the aliases of the harmonic coefficient $a_{\ell,m}$ are given by the following formula, 
\revise{\begin{align*}
	\alalm &=  \sum_{s=-\ell/2}^{Q-\ell-1}\sum_{r\in R_{m}^M\left(\ell+2s\right)} \zeta_{\ell,m} \zeta_{\ell+2s,m+2rM} I_{\ell,m}^{Q}\bra{\ell+2s,m+2rM} a_{\ell+2s,m+2rM} \\
	&+ \sum_{s\geq Q-\ell}\sum_{r\in R_{m}^{M;0}\left(\ell+2s\right)} \zeta_{\ell,m} \zeta_{\ell+2s,m+2rM} I_{\ell,m}^{Q}\bra{\ell+2s,m+2rM} a_{\ell+2s,m+2rM},
	\end{align*} }
where 
\begin{align*}
& \zeta_{\ell,m} = \left( \frac{2\ell+1}{2} \frac{ \left(\ell-m\right)!} {\left(\ell+m\right)!}\right) ^{\frac{1}{2}};\\
& I_{\ell,m}^{Q}\bra{\ell+2s,m+2rM} = \sum_{k=0}^{Q-1} w_k\sin \vartheta_k P_{\ell,m}\bra{\cos \vartheta_k} P_{\ell+2s,m+2rM}\bra{\cos \vartheta_k}.
\end{align*}
Simple algebraic manipulations show that this formula coincides with \eqref{eq:thmmain} claimed in Theorem \ref{thm:main} for $d=2$. 
\end{remark}
\revise{
\begin{remark}[Location and distance of the aliases]\label{rem:dist}
 From now on, we will define the location in the frequency domain of any harmonic coefficient $a_{\ell,\mcoord}$ as the set of numbers $\left\{\ell,\mcoord\right\}$. Following \cite{LiNorth}, we can thus define the distance in the frequency domain of the alias $a_{\ll^\prime,\mcoord^\prime}$ from the original coefficient by $\operatorname{dist}\left(\alm,\alalm\right)=\left\| \left(\ell-\ell^\prime,\mcoord-\mcoord^\prime\right) \right\|_{l^2}$,
where $\left\| \cdot \right\|_{l^2}$ is the Euclidean norm in the space of the square-summable sequences. If the uniform sampling scheme is considered, it follows that 
\begin{equation*} 
	\operatorname{dist}\left(\alm,a_{\ell+2s_0,\mcoord +2\mathbf{s}} \right) = \mathcal{D}\left(s_0,\mathbf{s}\right)= 2\sqrt{\sum_{i=0}^{d-2}s_i + rM}. 
\end{equation*}
Furthermore, from Theorem \ref{thm:main} it follows that $\mathcal{D}\left(s_0,\mathbf{s}\right)>2$. Indeed, the index $r$ can be null only if $s_0 \in B_0$ and, then, $s_0>0$. On the other hand, if there exists an alias with $S_0=\ldots=s_{d-2}=0$, then we have that $r>0$.
\end{remark}
The next result provides some practical rules on the choice of the parameters $Q_0,\ldots,Q_d-2,M$, with the aim to reduce the amount of aliases of a given harmonic coefficient $\alm$ assuming a uniform spherical sampling.
\begin{corollary}\label{cor:aliasing}
	Assume that Condition \ref{cond:uni} holds and that, furthermore, $Q_0 \geq \ldots \geq Q_{d-2}$, while $M>Q_0$. Thus, for any $\alm$, $\ell \in \mathbb{N}, \mcoord \in \mathcal{M}_\ell$, its aliases have locations $\left(\ell+2s_0, \mcoord+2\mathbf{s}\right)$, where $s_i \in B_i$ for $i=0,\ldots,d-2$ and $s_d-1=rM: r \in R_{m_{d-1}^M}$.   
\end{corollary}
\begin{remark}[Categories of locations]\label{rem:locationss}
	In view of Corollary \ref{cor:aliasing}, from now on we will denote the elements belonging to $\left\{s_0 \in B_0, \ldots, s_{d-2} \in B_{d-2}, r \in R_{m_{d-1}^M}\right\}$ as primary locations. The locations of the other aliases belonging to the set $\left\{s_0 \in D_{0}\left(\ell\right), \mathbf{s} \in Z_{\ell,\mcoord}^{\mathbf{Q}}\right\}\backslash\left\{s_0 \in B_0, \ldots, s_{d-2} \in B_{d-2}, r \in R_{m_{d-1}^M}\right\}$ will be labeled as secondary locations. According to Corollary \ref{cor:aliasing}, a proper choice of the sampling points can annihilate the aliases having secondary locations. The same does not hold for the ones in the primary locations. It is indeed impossible to remove all the aliases in primary locations just by choosing the sampling points and parameters. As we will discuss in Section \ref{sec:bandlimited}, these aliases can be completely erased, after a proper selection of sampling points, only if the random field is band-limited.\\
	Finally, note that under the assumptions of Corollary \ref{cor:aliasing}, it holds that 
	\begin{equation}\label{eqn:distdist} 
	\mathcal{D}\left(s_0,\mathbf{s}\right)\geq 2 Q. 
	\end{equation}	
\end{remark}
}  
\section{Aliasing for angular power spectrum}\label{sec:powerspec}
In this section, our purpose is to investigate on the aliasing effects as far as the spectral approximation of an isotropic random field is concerned. More specifically, we establish a method to identify the aliases of each element of the power spectrum $\left\{C_\ell: \ell \geq 0 \right\}$. \\ 
Assume to have an isotropic random field on $\Sd$, so that \eqref{eq:isotropy} and \eqref{eq:covariance} hold. When the integral \eqref{eqn:fieldcoeff} is replaced with the sum \eqref{eqn:aliaseries} under the Condition \ref{cond:uni}, we want to study how the aliasing errors arising in \eqref{eqn:aliaseries}, affect the estimation of $C_{\ell}=\Var{\alm}$ (see \eqref{eq:covariance}). In particular we are \revise{interested in} developing the presence of aliases when $C_\ell$ is approximated by the average
\begin{equation}\label{eq:spectrualias}
	\tilde{C}_{\ell}=\frac{1}{\Xi_d\left(\ell\right)} \sum_{\mcoord \in \Mset} \Var{\alalm},
\end{equation}
where $\Xi_d\left(\ell\right)$ is given by \eqref{eqn:sizespher} (cf, for example, \eqref{eq;clsum}). Let us recall that $D_0\left(\ell\right)$ is given by \eqref{eq:Dset}, and let $V_{\ell, '
	\mcoord} ^Q \left(\ell^\prime\right)$ be defined by
\begin{align*}
 &V_{\ell,
 	\mcoord} ^\mathbf{Q} \left(\ell^\prime\right) \\& =\!\!\! \sum_{\mathbf{s}\in Z_{\ell, \mcoord}^{\mathbf{Q}}}\! \prod_{j=1}^{d-1} \Anorm^2 h_{m_{j-1}+2s_{j-1},m_{j}+2s_{j};j}^2 \left(I_{m_{j\!-\!1},m_{j}}^{Q_{j\!-\!1}}\!\left(m_{j-1}\!+2s_{j-1},m_{j}\!+\!2s_{j}\right)\right)^2\!.
\end{align*}
Our findings, which extend to the $d$-dimensional sphere the outcomes of \cite[Theorem 3.1]{LiNorth} (cf. Remark \ref{rem:comp}), are produced in the following theorem.
\begin{theorem}\label{thm:power}
	Let $T$ be an isotropic random field on $\Sd$ with angular power spectrum given by \eqref{eq:covariance}. Under the assumption given in Condition \ref{cond:uni}, it holds that 
	\begin{equation*}
			\tilde{C}_{\ell} = \sum_{s_0 \in D_0\left(\ell\right)} \Lambda_\ell ^{Q}\left(\ell+2s_0\right) C_{\ell+2 s_{0}},
	\end{equation*}
	where 
	\begin{equation*}
	 \Lambda_\ell ^{Q}\left(\ell+2s_0\right) = \frac{1}{\Xi_{d}\left(\ell\right)} \sum_{\mcoord\in \Mset} V_{\ell, '
	 \mcoord} ^Q \left(\ell+2s_0\right).
	\end{equation*}
\end{theorem}
The proof of Theorem \ref{thm:power} can be found in Section \ref{sub:main}.
\section{Band-limited random fields} \label{sec:bandlimited}
 In this section, we establish the condition on the sample size,  leading to an exact reconstruction of the harmonic coefficients $\alm$ for band-limited random fields, in the paradigm of the spherical uniform design. In other words, for band-limited random fields and for a suitable choice of $\mathbf{Q}$, the approximation of the integral \eqref{eqn:fieldcoeff} by the sum \eqref{eq:firstalias} is exact and, then, there are no aliases, analogously to the findings described in \cite[Section 4]{LiNorth} for $d=2$. The reader is referred to Section \ref{sub:main} for the proofs of the theorems collected in this section.\\
 If the number of sampling points is sufficiently large with respect to the band-width characterizing the random field, we obtain two crucial results, stated in the next theorem. On the one hand, the band-limited random fields are alias-free in $\alalm$ and, on the other, they are exactly reconstructed by means of the Gaussian quadrature procedure described above. 
\begin{theorem}\label{thm:band}
	Assume that $T\left(x\right)$ is band-limited with bandwidth $L_0$, that is, the harmonic expansion given by \eqref{eq:randomfield} holds. If also Condition \ref{cond:uni} holds, with $Q=Q_0=\ldots=Q_{d-2}>L_0$ and $M>L_0$. Then, it holds that
	\begin{equation}\label{eqn:mainband}
		\alalm=\alm \quad \text{ for }\ell\leq L_0, \mcoord\in\Mset.
	\end{equation} 
Furthermore, for any $L \in \mathbb{N}$ satisfying $Q \geq L \geq L_0$, the following reconstruction holds exactly:
\begin{align}\notag
	T\left(x\right)& = \sum_{k_{0}=0}^{Q_{0}-1} \ldots \sum_{k_{d-1}=0}^{Q_{d-1}-1}  \left(\prod_{j=0}^{d-1} w_{k_{j}}^{\left(j+1\right)} \right)  \left(\prod_{j=1}^{d-1} \left(\sin \thetacord{j}_{k_{j-1}}\right)^{d-j} \right)\\
	&\cdot T\left(\thetacord{1}_{k_0},\ldots, \thetacord{d-1}_{k_{d-2}},\varphi_{k_{d-1}}\right)  \sum_{\ell=0}^{L} K_\ell \left(x,x_{\mathbf{k}} \right),\label{eqn:mainband2}
\end{align} 
where $x_{\mathbf{k}}=\left( \thetacoord_{k_0,\ldots,k_{d-2}} ,\varphi_{k_{d-1}}\right)$  and $ K_\ell$ is given by \eqref{eqn:addition}. 
\end{theorem} 
\begin{remark}
	In view of the results presented in Theorem \ref{thm:band}, the sample size $N$ has to satisfy the following condition, 
	\begin{align*}
		N \geq 2 L_0 ^d,
	\end{align*} 
	in order to avoid aliasing effects for band-limited random processes with band-width $L_0$. 
\end{remark}
\revise{
\begin{remark}
	If the random field is band-limited, the only possible aliases belong to secondary locations (see Remark \ref{rem:locationss}). Thus, a suitable choice of the parameters $Q_0,\ldots,Q_{d-2},M$ annihilates all the potential aliases. 
\end{remark}}
A random field has a band-limited power spectrum with bandwidth $P_L$ if $C_\ell = 0$ for any $\ell > P_L$. The following theorem shows that these random fields are aliases-free in $\tilde C_\ell$, employing a Gauss sampling under Condition \ref{cond:uni} and given a suitable sample size. 

\begin{theorem}\label{thm:bandcov}
	Let $T$ be a random field with a band-limited power spectrum with bandwidth $P_L$, sampled by means of a Gauss scheme under Condition \ref{cond:uni}, so that $Q = Q_0 = \ldots = Q_{d-2} \geq M > P_L$. Thus, it holds that
	\begin{equation*} 
		\Var{\alalm} = \Var{\alm} = C_\ell.
	\end{equation*}
\end{theorem}
\section{An example}\label{sec:example}
In this section, the reader is provided with an example, with the aim of giving a practical insight on the identification of the aliases of a harmonic coefficient. 
Let us fix $d=3$ and calculate the aliases of the harmonic coefficient $a_{0,0,0}$. Let us assume, furthermore, that $Q=Q_0 = Q_1$. 
\begin{figure}
		\fbox{	\includegraphics[width=\linewidth]{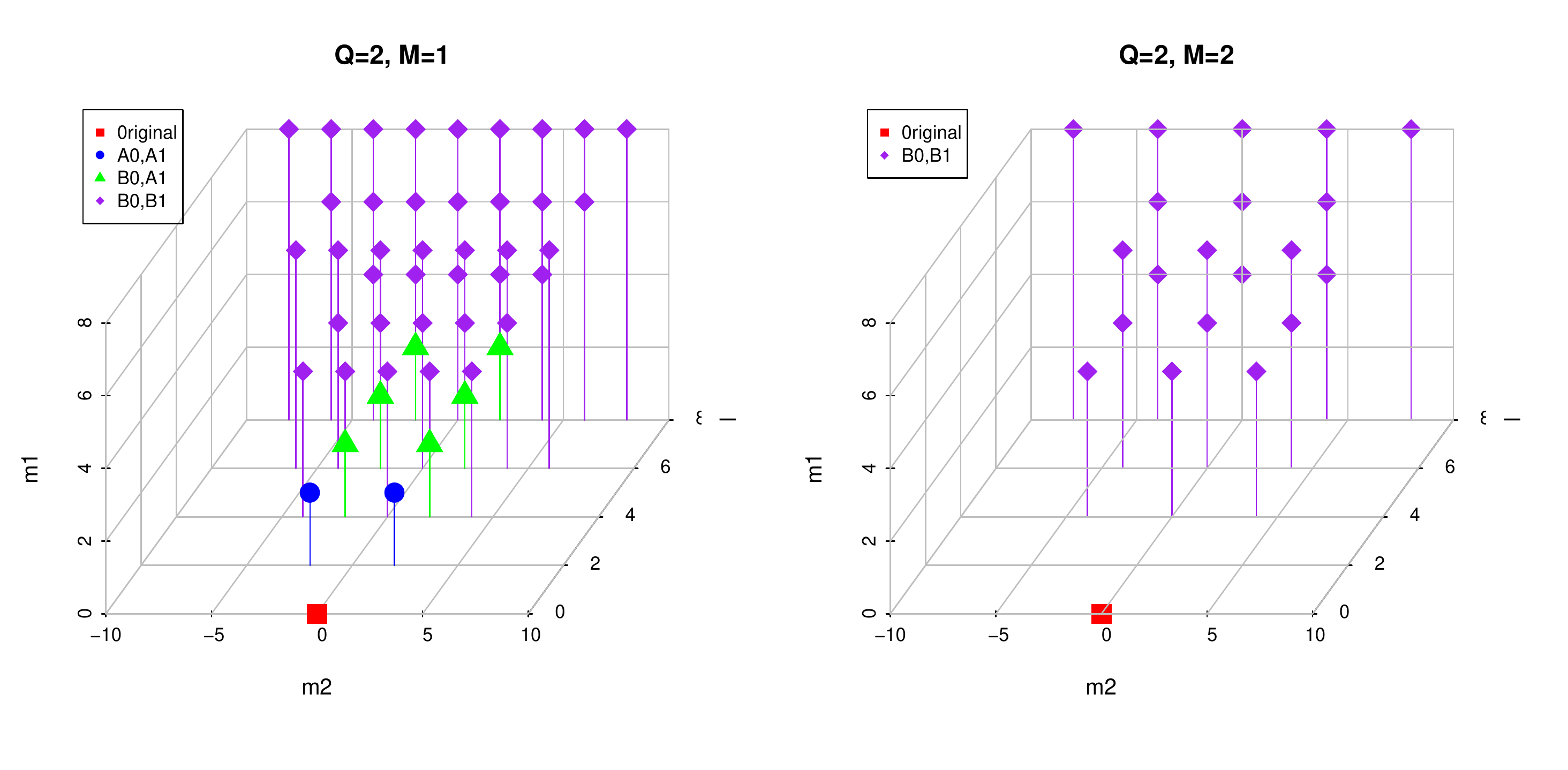}}
		\caption{Coordinates of the aliases for the coefficient $a_{0,0,0}$ for $Q=2$ and $M=1,2$. The left panel shows that this choice of $Q$ yields the presence of aliases with secondary locations. In the right panel these aliases are removed. Indeed, this choice of $Q$ annihilates all the aliases not featuring primary locations.}
		\label{fig:1}
		\end{figure}
We have that
	\begin{align*}
	\tilde{a}_{0,0,0} =&  a_{0,0,0}+ \sum_{s_0 \in D_0} \sum_{\left(s_1,s_2\right) \in Z_{0, 0,0}^{Q}} h_{0,0;1} h_{2s_0,2s_1;1}I_{0,0}^{Q}\left(2s_0,2s_{1}\right)\\ & \cdot  h_{0,0;2} h_{2s_1,2s_2;2}  I_{0,0}^{Q} \left(2s_{1},2s_{2}\right)  a_{2s_0, 2s_1,2s_2}. 
	\end{align*}
	\revise{On the one hand, using \eqref{eq:normalizing} to develop the intensity of the aliases, we obtain}
	\begin{align*}
	&  h_{0,0;1}  = \left(\frac{2}{\pi }\right)^{\frac{1}{2}};\\
	&  h_{0,0;2} = \frac{1}{\sqrt{2}};\\
	&  h_{2s_0,2s_1;1} =\left(\frac{2^{4s_1+1}\left(2s_0-2s_1 \right)!\left(2s_0 +1\right)\Gamma^2\left(2s_1+1\right) }{\pi \left(2s_0+2s_1+1\right)!}\right)^{\frac{1}{2}}\\ & \qqqquad= \left(\frac{2^{4s_1+1}\left(2s_0-2s_1 \right)!\left(2s_0 +1\right)\left(\left(2s_1\right)!\right)^2 }{\pi \left(2s_0+2s_1+1\right)!}\right)^{\frac{1}{2}};\\
	& h_{2s_1,2s_2;2} = \left(\frac{2^{4s_2-1} \left(2s_1-2s_2\right)! \left(4s_1 +1\right)\Gamma^2\left(2s_2+\frac{1}{2}\right) } {\pi \left(2s_1 +2s_2\right)!}\right)^{\frac{1}{2}}\\ & \qqqquad=\left(\frac{ \left(2s_1-2s_2\right)! \left(4s_1 +1\right)\left( \left(4s_2\right)!\right)^2 } { 2^{4s_2+1}\left(2s_1 +2s_2\right)!\left( \left(2s_2\right)!\right)^2}\right)^{\frac{1}{2}},
	\end{align*} 
	so that we can define
	\begin{align*}
	\epsilon_{s_0,s_1,s_2}=&h_{0,0;1}h_{2s_0,2s_1;1}h_{0,0;2}h_{2s_1,2s_2;2}\\
	=&\left(\frac{\left(2s_0-2s_1 \right)! \left(2s_1-2s_2\right)!\left(2s_0 +1\right)\left(4s_1 +1\right)}{ \left(2s_0+2s_1+1\right)! \left(2s_1 +2s_2\right)!}\right)^{\frac{1}{2}}\frac{2^{2\left(s_1-s_2\right)}\left(2s_1\right)!\left(4s_2\right)!}{\pi\left(2s_2\right)!}.
	\end{align*}
		\begin{figure}
		\fbox{\includegraphics[width=\linewidth]{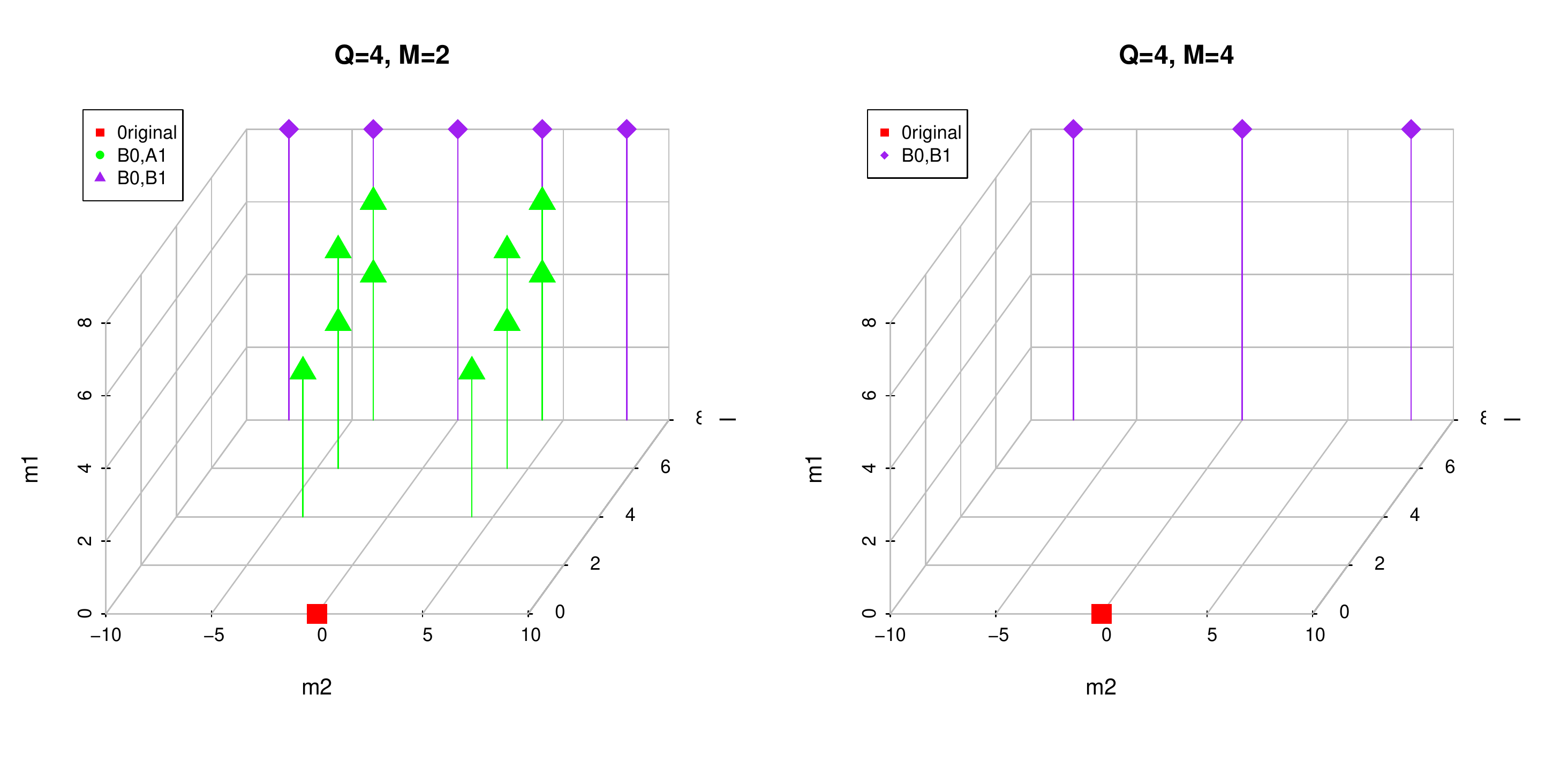}}
		\caption{Coordinates of the aliases for the coefficient $a_{0,0,0}$ for $Q=4$ and $M=1,2$. Also here, in the left panel aliases with secondary locations are detected, because of the choice of $Q$. The right panel features only aliases with primary locations. Indeed, this choice of $Q$ deletes all the aliases with secondary locations.}
		\label{fig:2}
	\end{figure}
			\begin{table}\centering
		\begin{tabular}{|c|c | c | c | }
			\hline 
			& \multicolumn{3}{|c|}{Q=2, M=1} \\ 
			\hline
			{$s_0$} & {$A_0$, $A_1$} & $B_0$, $A_1$& $B_0$, $B_1$\\
			\hline
			1 & $a_{2,2,-2}$, $a_{2,2,2}$ & &\\ 
			\hline
			\multirow{2}{*}{2} &  &{$a_{4,2,-2}$, $a_{4,2,2}$} &$a_{4,4,-4}$,$a_{4,4,-2}$, $a_{4,2,0}$,\\ & & &$a_{4,4,2}$,$a_{4,4,4}$\\ 
			\hline
			\multirow{4}{*}{3} &  &{$a_{6,2,-2}$, $a_{6,2,2}$} &$a_{6,4,-4}$,$a_{6,4,-2}$, $a_{6,2,0}$,\\ & & &$a_{6,4,2}$,$a_{6,4,4}$,$a_{6,6,-6}$,\\
			& & &$a_{6,6,-4}$,$a_{6,6,-2}$,$a_{6,6,0}$,\\ 
			& & &$a_{6,6,2}$,$a_{6,6,4}$,$a_{6,6,6}$\\	
			\hline	
			\multirow{7}{*}{4} &  &{$a_{8,2,-2}$, $a_{8,2,2}$} &$a_{8,4,-4}$,$a_{8,4,-2}$, $a_{8,4,0}$,\\ & & &$a_{8,4,2}$,$a_{8,4,4}$,$a_{8,6,-6}$,\\
			& & &$a_{8,6,-4}$,$a_{8,6,-2}$,$a_{8,6,0}$,\\ 
			& & &$a_{8,6,2}$,$a_{8,6,4}$,$a_{8,6,6}$,\\	
			& & &$a_{8,8,-8}$,$a_{8,8,-6}$,$a_{8,8,-4}$,\\
			& & &$a_{8,8,-2}$,$a_{8,8,0}$,$a_{8,8,2}$,\\ 
			& & &$a_{8,8,4}$,$a_{8,8,6}$,$a_{8,8,8}$, \\
			\hline 	
			\hline 
			& \multicolumn{3}{|c|}{Q=4, M=2} \\ 
			\hline
			{$s_0$} & {$A_0$, $A_1$} & $B_0$, $A_1$& $B_0$, $B_1$\\
			\hline
			1 & & & \\ 
			\hline
			2 &  &$a_{4,4,-4}$,$a_{4,4,4}$ & \\
			\hline
			3 &  &$a_{6,4,-4}$,$a_{6,4,4}$ & $a_{6,6,-4}$,$a_{6,6,4}$\\
			\hline	
			\multirow{2}{*}{4} &  &$a_{8,4,-4}$,$a_{8,4,4}$ & $a_{8,8,-8}$,$a_{8,4,-4}$,$a_{8,8,0}$ \\ & & &$a_{8,8,4}$,$a_{8,8,8}$\\
			\hline 	
		\end{tabular}
		\caption{List of aliases of the coefficient $a_{0,0,0}$ ($Q=2,4$ and $M=\frac{Q}{2}$),  for $s_0=1, \ldots 4$).}\label{tab:1}
	\end{table}
	On the other hand, we obtain from \eqref{eq:Dset}, \eqref{eq:A0},\eqref{eq:B0},\eqref{eq:Aj}, and \eqref{eq:Bj} that
	\begin{align*}
	& D_0 = \left\{ s_0 \in \integers:s_0\geq 0 \right\},\\ 
	& A_0=\left\{ s_0 \in \integers:0 \leq s_0\leq Q-1 \right\}, B_0=\left\{ s_0 \in \integers: s_0\geq Q-1 \right\},\\
	& H^{\left(1\right)}_0\left(2s_0\right) = \left\{ s_1 \in \integers: 0\leq s_1\leq s_0 \right\},\\ & A_1=\left\{ s_1 \in \integers: 0 \leq s_1\leq Q-1 \right\},\\& B_1=\left\{ s_1 \in \integers: Q-1 \leq s_1\leq s_0 \right\}.\\
	& R_{m_2}^M \left(2s_1\right)=\left\{r\in\integers: -\frac{s_{1}}{M}\leq r \leq \frac{s_{1}}{M} \right\}, 
	\end{align*}
	Hence, from \eqref{eqn:zeta} we have that 
	\begin{align*}
	Z_{0, 0,0}^{Q}= &\left\{\left(s_1, r \right) : s_1 \in \left( H^{\left(1\right);0}_{0}\left(2s_{0}\right) \ind{s_{0} \in A_{0}}+
	H^{\left(1\right)}_{0}\left(2s_{0}\right) \ind{s_{0} \in B_{0}}\right), \right.\\
	&  \left. r \in  \left( R^{M,0}_{0}\left(2s_1\right)\ind{s_{1} \in A_{1}} +  R^{M}_{0}\left(2s_1\right)\ind{s_1 \in B_1}\right) \right\}.
	\end{align*}
	We can then rewrite
	\begin{align}
	\tilde{a}_{0,0,0} \notag
	& =  a_{0,0,0}+  \sum_{s_0 =0}^{Q-1}  \sum_{s_1 = 1}^{s_0}  \sum_{\substack{r= -\frac{s_1}{M}\\s_2\neq 0 }} ^{\frac{s_1}{M}}\! \!\!\!\!\epsilon_{s_0,s_1,rM} I_{0,0}^{Q}\left(2s_0,2s_{1}\right)  I_{0,0}^{Q} \left(2s_{1},2rM\right)  a_{2s_0, 2s_1,2rM} \\
	&+\sum_{s_0 \geq Q} \left( \sum_{s_1 = 0}^{Q-1} \sum_{\substack{r= -s_1/M\\s_2\neq 0 }} ^{s_1/M}  \epsilon_{s_0,s_1,rM} I_{0,0}^{Q}\left(2s_0,2s_{1}\right)  I_{0,0}^{Q} \left(2s_{1},2rM\right) \right.\notag\\
	& + \left. \sum_{s_1 = Q}^{s_0}  \sum_{r = -s_1/M} ^{s_1/M} \epsilon_{s_0,s_1,rM}  I_{0,0}^{Q}\left(2s_0,2s_{1}\right)  I_{0,0}^{Q} \left( 2s_{1},2rM \right)  \right) a_{2s_0, 2s_1,2rM}. \label{eq:third}
	\end{align}
	\begin{table}\centering
	\begin{tabular}{|c|c | c | c | }
		\hline 
		& \multicolumn{3}{|c|}{Q=2, M=2} \\ 
		\hline
		{$s_0$} & {$A_0$, $A_1$} & $B_0$, $A_1$& $B_0$, $B_1$\\
		\hline
		1 & & & \\ 
		\hline
		2 & & & $a_{4,4,-4}$,$a_{4,4,-2}$, $a_{4,2,0}$\\
		\hline
		\multirow{2}{*}{3} & & &$a_{6,4,-4}$,$a_{6,4,0}$, $a_{6,4,4}$, \\ & & &$a_{6,6,-4}$,$a_{6,6,0}$, $a_{6,6,4}$  \\
		\hline	
		\multirow{4}{*}{4} & & &$a_{8,4,-4}$,$a_{8,4,0}$, $a_{8,4,4}$ \\ & & & $a_{8,6,-4}$,$a_{8,6,0}$, $a_{8,6,4}$\\
		& & &$a_{8,8,-8}$,$a_{8,8,-4}$, $a_{8,8,0}$ \\ 
		& & & $a_{8,8,4}$,$a_{8,8,8}$ \\	
		\hline 	
		\hline 
		& \multicolumn{3}{|c|}{Q=4, M=2} \\ 
		\hline
		{$s_0$} & {$A_0$, $A_1$} & $B_0$, $A_1$& $B_0$, $B_1$\\
		\hline
		1 & & & \\ 
		\hline
		2 & & & \\
		\hline
		3 & & & \\ 
		\hline	
		4 &  & &$a_{8,8,-8}$,$a_{8,4,0}$,$a_{8,8,8}$ \\
		\hline 	
	\end{tabular}
	\caption{List of aliases of the coefficient $a_{0,0,0}$ ($Q=2,4$ and $M=Q$),  for $s_0=1, \ldots 4$).}\label{tab:2}
\end{table}
	Observe that the first line in \eqref{eq:third} describes the aliases obtained for $s_0 \in A_0$, while the other two lines contain the aliases corresponding to $s_0 \in B_0$. Notice that if $s_0 \in A_0$, then $B_1=\emptyset$. As a consequence, it follows that both the indexes $s_1$ and $s_2$ can not take the null-value. When $s_0 \in B_0$, we have that $A_1= \left\{0,\ldots,Q-1\right\}$ and $B_1=\left\{Q,\ldots,s_0\right\}$. Hence, we obtain the second and the third sums in \eqref{eq:third}.\\
	 
	\revise{We want to establish here the locations of the aliases that affect $a_{0,0,0}$ for some choices of $Q$ and $M$. Let us take $Q=2,4$ and $M=Q/2,Q$. Here, for the sake of the computational simplicity, we will take into account only $s_0 = 1,\ldots, 4$. All the aliases of $a_{0,0,0}$ for the considered range of $s_0$ are collected in Table \ref{tab:1} and in Table \ref{tab:2}. For any choice of $Q$ and $M$, each column contains aliases belonging to the sets $\left\{s_0 \in A_0, s_1 \in A_1\right\}$, $\left\{s_0 \in B_0, s_1 \in A_1\right\}$, and $\left\{s_0 \in B_0, s_1 \in B_1\right\}$ respectively. The locations of the aliases are also shown in Figure \ref{fig:1}, for $Q=2$, and Figure \ref{fig:2}, for $Q=4$.\\ According to the results here produced, we can notice that
		\begin{itemize} 
			\item the minimum distance of the aliases increases when $Q$ grows, following Remark \ref{rem:dist} and Equation \eqref{eqn:distdist} in Remark \ref{rem:locationss}; 
			\item all the aliases with secondary locations (see Remark \ref{rem:locationss}), belonging thus to the subsets $\left\{s_0\in A_0, s_1 \in A_1 \right\}$ and $\left\{s_0\in B_0, s_1 \in A_1 \right\}$, vanish for $M=Q$, as stated in Corollary \ref{cor:aliasing};
			\item the coefficient $a_{0,0,0}$ is not affected by aliasing if it is the harmonic coefficient of a band-limited function with band width $L_0< Q$, as stated in Theorem \ref{thm:band}.
	\end{itemize}	
	}

\section{Proofs}\label{sec:proofs}
In this section, we provide proofs for the main and auxiliary results.
\subsection{Proofs of the main results} \label{sub:main}
\begin{proof}[Proof of Theorem \ref{thm:auxiliary}]
	Using \eqref{eq:ffunc}, \eqref{eqn:spharm}, \eqref{eq:weight} and \eqref{eq:samplingangles} in \eqref{eqn:aliasfunct} yields
	\begin{align*}
	&\Alias \\
	  &\quad = \sum_{k_{0}=0}^{Q_{0}-1} \ldots \sum_{k_{d-1}=0}^{Q_{d-1}-1}  \left(\prod_{j=1}^{d} w_{k_{j-1}}^{\left(j\right)} \right)  \left(\prod_{j=1}^{d-1} \left(\sin \thetacord{j}_{k_{j-1}}\right)^{d-j} \right)\\
	&\quad \quad \cdot \left( \frac{1}{\sqrt{2\pi}} \prod_{j=1}^{d-1} \left( \Anormprime 
	\gegenbauer{m^{\prime}_{j-1}}{m^{\prime}_{j}}{j}\left(\cos \thetacord{j}_{k_{j-1}}\right) \left(\sin \thetacord{j}_{k_{j-1}}\right)^{m^{\prime}_{j}} \right)e^{im_{d-1}^{\prime}\varphi_{k_{d-1}}}\right) \\ 
	&\quad \quad\cdot \left( \frac{1}{\sqrt{2\pi}} \prod_{j=1}^{d-1} \left( \Anorm 
	\gegenbauer{m_{j-1}}{m_{j}}{j} \left(\cos \thetacord{j}_{k_{j-1}}\right)\left(\sin \thetacord{j}_{k_{j-1}}\right)^{m_{j}} \right)e^{im_{d-1}\varphi_{k_{d-1}}} \right) \\
	&=\quad \frac{1}{2\pi} \prod_{j=1}^{d-1} \left( 
	\sum_{k_{j-1}=0}^{Q_{j-1}-1}  w_{k_{j-1}}^{\left(j\right)}   \left(\sin \thetacord{j}_{k_{j-1}}\right)^{m_j+m^\prime_{j} +d-j}  \Anorm \Anormprime  \right. \\ & \quad \quad \left. \cdot \gegenbauer{m_{j-1}}{m_{j}}{j} \left(\cos \thetacord{j}_{k_{j-1}}\right)  \gegenbauer{m^\prime_{j-1}}{m^\prime_{j}}{j} \left(\cos \thetacord{j}_{k_{j-1}}\right)\right)\\
	& \quad \quad \cdot \left(\sum_{k_{d-1}=0}^{Q_{d-1}-1}w_{k_{d-1}}^{\left(d\right)}  e^{i\left(m_{d-1}^{\prime}-m_{d-1} \right)\varphi_{k_{d-1}}}  \right),
	\end{align*}
	as claimed.
\end{proof}

\begin{proof}[Proof of Theorem \ref{thm:main}]
	We divide this proof in two parts. The first part establishes explicit bounds for the indices $s_0,\ldots,s_{d-2},r$ by means of 
	\begin{enumerate}
		\item the parity properties of the Gegenbauer polynomials (see Lemma \ref{lemma:simmetry});
		\item the definition of $\Mset$ (cf. \eqref{eq:Mset}), which exploits the definition of spherical harmonics in \eqref{eqn:spharm}.
	\end{enumerate} 
	The second part of the proof detects then some sets of indices $s_0,\ldots,s_{d-2},r$ for which $\Alias=0$ as a consequence of
	\begin{enumerate}
		\item the order of the quadrature formula (see \eqref{eq:quadrat}).	
		\item  the orthogonality of the Gegenbauer polynomials (see \eqref{eq:gegenortho}); 
	\end{enumerate}
	For both cases, we follow a backward induction step, studying first the aliasing effects due to the trapezoidal sampling for coordinate $j=d$, using the results holding for the $j$-th component to prove the statement for the $j-1$-th component, until we reach $j=1$. \\
	
	\noindent \textit{Part 1 - } Here our purpose it to exploit either properties due to the uniform sampling and the ones related to the harmonic numbers of spherical harmonics, to establish lower and, where possible, upper bounds for the indices $s_0, \ldots, s_{d-2}, r$. These indices identify the aliases of the harmonic coefficient $\alm$, given in the form $a_{\ell+2s_0, \mcoord +2\mathbf{s}}$.\\ Let us consider initially $j=d$ and apply to the coordinate $\varphi$ the standard trapezoidal rule. As well as in \cite{LiNorth} (see also \cite{dettedsphere}), using \eqref{eq:azimuthpoint} and \eqref{eq:azimuthweight} in \eqref{eq:Jfunc} yields
	\begin{equation}\label{eqn:Junif} 
	\Jfunczwei = \frac{\pi}{M} \sum_{q=0}^{2M-1} e^{i\bra{m_{d-1}^{\prime}-m_{d-1}}\frac{q\pi}{M}}=2\pi \delta_{m_{d-1}+2rM}^{m_{d-1}^\prime},
	\end{equation}
	where $r \in \mathbb{Z}$ is such that $\abs{m_{d-1}+2rM}\leq m_{d-2}^{\prime}$.
	Indeed, from \eqref{eqn:spharm} it follows that $Y_{\ell^\prime,\mcoord^\prime} \left(x\right)$ is well-defined only for $\abs{m^\prime_{d-1}}\leq m_{d-2}^\prime$. Thus, it holds that
	$r \in R^M_{m_{d-1}} \bra{m_{d-2}^\prime}$, where
	\begin{equation*}
	R^M_{m_{d-1}} \bra{m_{d-2}^\prime}:=\bbra{r\in \integers :-\frac{m_{d-2}^\prime +m_{d-1}}{2M}\leq r \leq \frac{m_{d-2}^{\prime}-m_{d-1}}{2M}}.
	\end{equation*}
	\noindent Consider now $j=d-1$. The component $\thetacord{d-1}$ is subject to the aforementioned Gauss-Legendre quadrature formula (cf. the case $d=2$ in \cite{LiNorth}). Indeed, by using \eqref{eqn:Junif} jointly with the definition of the sampling points and weights given by \eqref{eq:thetasamp} and \eqref{eq:thetawei} respectively with $j=d-1$, the $(d-1)$-th aliasing factor is given by 
	\begin{align}\notag
	& I_{m_{d-2},m_{d-1}}^{Q_{d-2}} \left(m^\prime_{d-2},m_{d-1}+2rM\right)\\ & \quad \quad =  \sum_{k_{d-2}=0}^{Q_{d-2}-1} w_{k_{d-2}}^{\left(d-1\right)} \left( \sin \vartheta_{k_{d-2}}^{\left(d-1\right)} \right)^{2\left(m_{d-1}+rM\right)+1} C_{m_{d-2}-m_{d-1}}^{\left(m_{d-1}+\frac{1}{2}\right)} \left(\cos \thetacord{d-1}_{k_{d-2}}\right)\notag\\& \quad \quad \quad \cdot C_{m^{\prime}_{d-2}-m_{d-1} - 2rM }^{\left(m_{d-1} + 2rM + \frac{1}{2}\right)} \left(\cos \thetacord{d-1}_{k_{d-2}}\right).\label{eq:dmenodue}
	\end{align}
	Observe now that the Legendre polynomials can be expressed in terms of a Gegenbauer polynomial by means of the formula
	\begin{align*}
	& P_{m_{d-2},m_{d-1}}\left(\cos \thetacord{d-1}_{k_{d-2}}\right) \\
	& \quad \quad \quad=\frac{\left(2 m_{d-1}\right)!}{2^{m_{d-1}}\left( m_{d-1}\right)!} \left(\sin \thetacord{d-1}_{k_{d-2}}\right)^{m_{d-1}}C_{m_{d-2}-m_{d-1}}^{\left(m_{d-1}+\frac{1}{2}\right)} \left(\cos \thetacord{d-1}_{k_{d-2}}\right),
	\end{align*}
	see for example \cite[Formula 4.7.35]{szego}.
	Hence, we obtain that 
	\begin{align}\notag
	& I_{m_{d-2},m_{d-1}}^{Q_{d-2}} \left(m^\prime_{d-2},m_{d-1}+2rM\right) \\& \quad \quad \quad =  c_{m_{d-1}}c_{m_{d-1}+2rM}   \sum_{k_{d-2}=0}^{Q_{d-2}-1} w_{k_{d-2}}^{\left(d-1\right)}  \sin \vartheta_{k_{d-2}}  P_{m_{d-2},m_{d-1}}\left(\cos \thetacord{d-1}_{k_{d-2}}\right)\notag\\ & \quad \quad \quad\quad \cdot P_{m^\prime_{d-2},m_{d-1}+2rM}\left(\cos \thetacord{d-1}_{k_{d-2}}\right)\label{eq:parity1},
	\end{align}
	where 
	\begin{equation*}
	c_m = \left(\frac{\left(2 m\right)!}{2^{m}\left( m\right)!} \right)^{-1}.
	\end{equation*}
	In analogy to \cite[Theorem 2.1]{LiNorth}, using \eqref{eq:parity2}, given in Lemma \ref{lemma:simmetry}, for $j=d-1$, in \eqref{eq:parity1}
	leads to
	\begin{equation*}
	I_{m_{d-2},m_{d-1}}^{Q_{d-2}} \left(m^\prime_{d-2},m_{d-1}+2rM\right) = 0 \quad \text{for any } m_{d-1}^\prime =  m_{d-2} + 2 s_{d-2}+1, s_{d-2} \in \mathbb{N}_{0}.
	\end{equation*} 
	In other words, the $d-1$-th aliasing factor is not null only for even values of  $\abs{m_{d-2}^\prime-m_{d-2}}$, that is,
	\begin{equation*}
	m_{d-2}^\prime = m_{d-2} + 2 s_{d-2}, 
	\end{equation*}
	where $s_{d-2}\in D_{m_{d-2}}$, given by
	\begin{equation*} 
	D_{m_{d-2}} = \left\{s_{d-2}\in \integers: s_{d-2} \geq -\frac{m_{d-2}}{2}\right\},
	\end{equation*}
	which guarantees that $m_{d-2}^\prime\geq 0$ and, thus, a well-defined aliasing factor in \eqref{eq:dmenodue}. \\
	On the one hand, using $m_{d-2}^\prime = m_{d-2}+2s_{d-2}$ in the set concerning the $d$-th aliasing factor, we have that $r \in R_{m_{d-1}}^M\left(m_{d-2}+2s_{d-2}\right)$, as given by \eqref{eq:Rfunc}. \\
	On the other hand, following \eqref{eq:Mset} and \eqref{eqn:spharm},  it holds that $m_{d-2}^\prime=m_{d-2}+2{s_{d-2}}\leq m_{d-3}^\prime$. Thus, $s_{d-2} \in R_{m_{d-2}}\left(m_{d-3}^\prime\right)$,
	where
	\begin{equation*}
	R_{m_{d-2}}\left(m_{d-3}^\prime\right)=\left\{s_{d-2} \in \integers : s_{d-2} \leq \frac{m_{d-3}^\prime - m_{d-2}}{2} \right\}.
	\end{equation*}
	Therefore we obtain that $s_{d-2} \in H^{\left(d-2\right)}_{m_{d-2}}\left(m_{d-3}^\prime\right)$,
	where 
	\begin{equation*}
	H^{\left(d-2\right)}_{m_{d-2}}\left(m_{d-3}^\prime\right)=D_{m_{d-2}}\cap R_{m_{d-2}}\left(m_{d-3}^\prime\right).
	\end{equation*}
	\noindent Consider now $2 \leq j \leq d-2$. For each component, we use a suitable Gauss-Gegenbauer quadrature rule described above (see also \cite[Lemma 3.1]{dettedsphere}). 
	Using Lemma \ref{lemma:simmetry} yields the following outcome. 
	If $I^{Q_{j}}_{m_j,m_{j+1}} \left(m^\prime_j,m^\prime_{j+1}\right) \neq 0$ only when $m_j^\prime=m_j+2s_j$, for $s_j \in H^{\left(j+1\right)}_{m_{j}}\left(m_{j-1}^\prime\right)$, then $\Ifunc \neq 0$ only when  $m_{j-1}^\prime=m_{j-1}+2s_{j-1}$, $s_{j-1} \in H^{\left(j\right)}_{m_{j-1}}\left(m_{j-2}^\prime\right)$. \\
	On the one hand, Formula \eqref{eq:parity3} in Lemma \ref{lemma:simmetry} with $m^\prime_{j}=m_{j}+2s_j$ yields  $I^{Q_{j-1}}_{m_{j-1},m_{j}} \left(m^\prime_{j-1},m_{j}+2s_j\right) \neq 0$ only for $m^\prime_{j-1}=m_{j-1}+2s_{j-1}$, so that the aliases with respect to the $j$-th component are identified by the function
	\begin{align*}
	I^{Q_{j-1}}_{m_{j-1},m_{j}} &\left(m_{j-1}+2s_{j-1},m_{j}+2s_j\right)\\  =  & \sum_{k_{j-1}=0}^{Q_{j-1}-1} w_{k_{j-1}}^{\left(j\right)} \left( \sin \vartheta_{k_{j-1}}^{\left(j\right)} \right)^{2\left(m_{j}+s_j\right)+d-j}  C_{m_{j-1}-m_{j}}^{\left(m_{j}+\frac{d-j}{2}\right)} \left(\cos \thetacord{j}_{k_{j-1}}\right) \\ &\cdot C_{m_{j-1}+2s_{j-1}-\left(m_{j} + 2s_j\right) }^{\left(m_{j} + 2s_j + \frac{d-j}{2}\right)} \left(\cos \thetacord{j}_{k_{j-1}}\right).
	\end{align*}
	It is straightforward to set $s_{j-1} \in D_{m_{j-1}}$, where 
	\begin{equation*}
	D_{m_{j-1}}=\left\{s_{j-1}\in \mathbb{Z}:s_{j-1}\geq -\frac{m_{j-1}}{2}  \right\}, 
	\end{equation*}
	so that the polynomials  in $I^{Q_{j-1}}_{m_{j-1},m_{j}} \left(m_{j-1}+2s_{j-1},m_{j}+2s_j\right)$,
	\begin{align*}
	&\omega_{k_{j-1}}^{\left(j\right)} \left( 1-t_{k_{j-1}}^{\left(j\right)} \right)^{\left(m_{j}+s_j\right)}  C_{m_{j-1}-m_{j}}^{\left(m_{j}+\frac{d-j}{2}\right)} \left(t^{j}_{k_{j-1}}\right) C_{m_{j-1}+2s_{j-1}-\left(m_{j} + 2s_j\right) }^{\left(m_{j} + 2s_j + \frac{d-j}{2}\right)} \left(t^{j}_{k_{j-1}}\right)\\
	& \qquad\quad =w_{k_{j-1}}^{\left(j\right)} \left( \sin \vartheta_{k_{j-1}}^{\left(j\right)} \right)^{2\left(m_{j}+s_j\right)+d-j}  C_{m_{j-1}-m_{j}}^{\left(m_{j}+\frac{d-j}{2}\right)} \left(\cos \thetacord{j}_{k_{j-1}}\right) \\ 
	& \qquad \quad \cdot C_{m_{j-1}+2s_{j-1}-\left(m_{j} + 2s_j\right) }^{\left(m_{j} + 2s_j + \frac{d-j}{2}\right)} \left(\cos \thetacord{j}_{k_{j-1}}\right)
	\end{align*}
	is of degree $m_{j-1} +2s_{j-1}\geq0$.\\
	On the other hand, taking into account \eqref{eq:Mset} and \eqref{eqn:spharm}, it follows that $m_{j-1}^\prime = m_{j-1} + 2s_{j-1} \leq m^{\prime}_{j-2}$. Thus we obtain that  $s_{j-1} \in R_{m_{j-1}}\left(m^{\prime}_{j-2}\right)$, where
	\begin{equation*}
	R_{m_{j-1}}\left(m^{\prime}_{j-2}\right)= \left\{s_{j-1} \in \mathbb{Z}:  s_{j-1} \leq \frac{m^{\prime}_{j-2}-m_{j-1}}{2} \right\},
	\end{equation*}
	with $m_{j-2}^\prime=m_{j-2}+2s_{j-2}$. Combining these two results and recalling \eqref{eq:Hset}, for $j=2,\ldots,d-1$, it holds that
	\begin{equation*}
	s_{j-1} \in H^{\left(j-1\right)}_{m_{j-1}}\left(m_{j-2}^\prime \right), \quad \text{where } 	H^{\left(j-1\right)}_{m_{j-1}}\left(m^\prime_{j-2}\right)= D_{m_{j-1}} \cap  R_{m_{j-1}}\left(m^{\prime}_{j-2}\right).
	\end{equation*}
	Furthermore, the following step of the backward procedure yields $m_{j-2}^\prime=m_{j-2}+2s_{j-2}$, so that 
	\begin{equation*}
	s_{j-1} \in H^{\left(j-1\right)}_{m_{j-1}}\left(m_{j-2}+2s_{j-2}\right), 
	\end{equation*}
	for $j=2,\ldots,d-1$. 
	\noindent Consider, finally, the case $j=1$. This aliasing factor is given by
	\begin{equation*}
	I_{\ell,m_1}^{Q_{0}}\left(\ell^\prime,m_1+2s_1\right) \quad \text{ for } s_1 \in H^{\left(1\right)}_{m_1}\left(\ell^\prime\right). 
	\end{equation*}
	Here we can thus select $\ell^\prime=\ell+2s_0$, $s_0 \in D_0\left(\ell\right)$, where $D_0\left(\ell\right)$ is given by \eqref{eq:Dset}. Note that $s_0$ is the only index that is not selected from a set of finitely many elements.\\
	
	\noindent \textit{Part 2 - } Here our aim is to use the order of the used quadrature formula to convert, when possible, the sums of $\Ifunc$ to integrals. Then, we exploit the orthogonality of the Gegenbauer polynomials (see Section \ref{sec:prel}) to establish further combinations of indices $s_0,\ldots,s_{d-1},r$ which lead to a null aliasing function.\\
	First of all, for any $j=1,\ldots,{d-1}$, as stated in Remark \ref{rem:preliminary}, the following decomposition holds
	\begin{align*} 
	&D_0\left(\ell\right) = A_0 \cup B_0,\\
	&H^{\left(j\right)}_{m_j}\left(m_{j-1}+2s_{j-1}\right) = A_{j}\cup B_{j},
	\end{align*}
	where $A_0$, $B_0$, $A_j$, and $B_j$ are given by \eqref{eq:A0}, \eqref{eq:B0}, \eqref{eq:Aj}, and \eqref{eq:Bj} respectively. Recall also that $A_j$ and $B_j$ are defined by \eqref{eq:Ajbis}, and \eqref{eq:Bjbis} if $s_{j-1}\leq Q_{j}-\frac{m_{j-1}+m_{j}}{2}$. \\
	\noindent Now, let $h_{d-2}:\sbra{-1,1}\rightarrow \reals$ be a polynomial function of degree strictly smaller than $2Q_{d-2}$; hence, by using the aforementioned Gauss-Legendre quadrature formula (of order $2Q_{d-2}$) we obtain that 
	\begin{align}
	\sum_{k_{d-2}=0}^{Q_{d-2}-1} w_{k_{d-2}}^{\left(d-1\right)}  \sin \vartheta_{k_{d-2}}^{\left(d-1\right)} h_{d-2}\left(\cos \thetacord{d-1}_{k_{d-2}}\right) & = \sum_{k_{d-2}=0}^{Q_{d-2}-1} \omega_{k_{d-2}}^{\left(d-1\right)} h_{d-2}\bra{t_p}\notag\\
	& = \int_{-1}^{1}h_{d-2}\bra{t} \diff t. \label{eqn:lago}
	\end{align}
	As a straightforward consequence, (cf. \cite[Section 2.2]{LiNorth}), for $0 \leq m_{d-2} \leq \left(Q_{d-2}-1\right)$ and $s_{d-2} \in \mathbb{Z}\cap \left[-m_{d-2}/2, Q_{d-2}-m_{d-2}-1\right]$, \eqref{eqn:lago} holds with $h_{d-2}\left( t\right) =P_{m_{d-2},m_{d-1}}\left( t\right) P_{m_{d-2}+2s_{d-2},m_{d-1}}\left( t\right)$, a polynomial of degree smaller than  $2Q_{d-2}$. Hence, we obtain that 
	\begin{align*}
	\notag I_{m_{d-2},m_{d-1}}^{Q_{d-2}} \left(m_{d-2}+2s_{d-2},m_{d-1}\right)=& \int_{-1}^{1} P_{m_{d-2},m_{d-1}}\left(t\right) P_{m_{d-2}+2s_{d-2},m_{d-1}}\left(t\right) \diff t \\=&\left( \frac{\left(m_{d-2}-m_{d-1} \right)! }{ \left(m_{d-2}+m_{d-1}\right)} \frac{\left(2m_{d-2} + 1\right)} {2}\right)^{-1}.
	\delta_{s_{d-2}}^{0} 
	\end{align*} 
	Hence, in the uniform sampling approach, all the aliases of $\alm$ corresponding to the values $r=0$ and $-m_{d-2}/2 \leq s_{d-2} \leq Q_{d-2}-m_{d-2}$, $s_{d-2} \neq 0$, are annihilated. Aliases of $\alm$ exist for the following combinations of the indices $s_{d-2},r$:
	\begin{itemize}
		\item $s_{d-2}\in A_{d-2}$ and $r \in  R^{M;0}_{m_{d-1}}\left(m_{d-2}+2s_{d-2}\right)$;
		\item  $s_{d-2}\in B_{d-2}$ and $r \in  R^{M}_{m_{d-1}}\left(m_{d-2}+2s_{d-2}\right)$,
	\end{itemize}
	where $R^{M;0}_{m_{d-1}}\left(m_{d-2}+2s_{d-2}\right)$ is given by \eqref{eq:R0}. Thus, if we define $s_{d-1}=rM$, it holds that $s_{d-1} \in \Delta_{d-1}$, where $\Delta_{d-1}$ is defined by \eqref{eq:Dlast}.\\
	Take now $1\leq j \leq d-2$ and let $h_{j-1}:\sbra{-1,1}\rightarrow \reals$ be a polynomial function of degree strictly smaller than $2Q_{j-1}$. The Gauss-Gegenbauer quadrature rule leads thus to
	\begin{align}
	\sum_{k_{j-1}=0}^{Q_{j-1}-1} w_{k_{j-1}}^{\left(j\right)}\left(  \sin \vartheta_{k_{j-1}}^{\left(j\right)}\right)^{d-j} h_{j-1}\left(\cos \thetacord{j}_{k_{j-1}}\right)& = \sum_{k_{j-1}=0}^{Q_{j-1}-1}  \omega_{k_{j-1}}^{\left(j\right)} h _{ j-1}\bra{ t_{k_{j-1}}^{\left(j\right)} } \notag\\ & \label{eqn:lago1}= \int_{-1}^{1}h_{j-1}\bra{t}\diff t.
	\end{align}
	Then, for $0 \leq m_{j-1}  \leq \left(Q_{j-1}-1\right)$ and $s_{j-1} \in \mathbb{Z}\cap \left[\left.-m_{j-1}/2, Q_{j-1}-m_{j-1}\right)\right.$, \eqref{eqn:lago1} holds with 
	\begin{equation*}
	h_{j-1}\left( t\right) = \left( 1-t^2\right)^{\left(m_{j}+s_j\right)}  C_{m_{j-1}-m_{j}}^{\left(m_{j}+\frac{d-j}{2}\right)} \left(t\right) C_{m_{j-1}+2s_{j-1}-\left(m_{j} + 2s_j\right) }^{\left(m_{j} + 2s_j + \frac{d-j}{2}\right)} \left(t\right),
	\end{equation*} a polynomial of degree $ 2\left(m_{j-1}+s_{j-1}\right)<2Q_{j-1}$. Hence, from the orthogonality of the Gegenbauer polynomials (cf. \eqref{eq:gegenortho}), it follows that
	\begin{align}
	\notag 	& I^{Q_{j-1}}_{m_{j-1},m_{j}} \left(m_{j-1}+2s_{j-1},m_{j}\right) \\& \qqqquad =   \int_{-1}^{1} \gegenbauer{m_{j-1}}{m_{j}}{j}\left(t\right) \gegenbauer{m_{j-1}+2s_{j-1}}{m_{j}}{j}\left(t\right) \left(1-t^2\right)^{m_j+\frac{d-j-1}{2}} \notag \\ & \qqqquad =\frac{\pi 2^{1-2\left(m_j+\frac{d-j}{2}\right)} \Gamma\left(m_{j-1}+m_{j}+d-j\right)}{\left(m_{j-1}-m_{j}\right)!\left(m_{j-1}+\frac{d-j}{2}\right)\Gamma^2\left(\left(m_j+\frac{d-j}{2}\right)\right)} 
	\delta_{s_{j-1}}^{0}\label{eqn:Id}.
	\end{align} 
	Thus, $I^{Q_{j-1}}_{m_{j-1},m_{j}} \left(m_{j-1}+2s_{j-1},m_{j}\right)$ is annihilated for $s_j=0$ and $-m_{j-1}/2 \leq s_{j-1} \leq Q_{j-1}-m_{j-1}$, $s_{j-1} \neq 0$. For any $j=1,\ldots,d-2$, aliases $a_{\ell+s_0,\mcoord+\mathbf{s}}$ exist for
	\begin{itemize}
		\item $	s_{j-1} \in A_{j-1}$
		and $s_{j} \in H^{\left(j\right);0}_{m_{j}}\left(m_{j-1}+2s_{j-1}\right)$ ;
		\item $s_{j-1}\in B_{j}$ and $s_{j} \in H^{\left(j\right)}_{m_{j}}\left(m_{j-1}+2s_{j-1}\right)$, 
	\end{itemize}
	where $H^{\left(j\right);0}_{m_{j}}\left(m_{j-1}+2s_{j-1}\right)$ is given by \eqref{eq:H0}. In other words, for any $j=1,\ldots,d-2$, it holds that $s_{j} \in \Delta_{j}$, where $\Delta_{j}$ is defined by \eqref{eq:Dother}.\\
	\noindent Recombining all these results for $j=1,\ldots,d$ yields the fact that the aliases $a_{\ell+2s_0, \mcoord +2\mathbf{s}}$ exist for $\mathbf{s} \in Z^{\mathbf{Q}}_{\ell, \mcoord}$, where $Z^{\mathbf{Q}}_{\ell, \mcoord}$ is defined by \eqref{eqn:zeta}, as well as for $s_0 \in D_0\left(\ell\right)$ (cf. Part 1), as claimed.
\end{proof}
\begin{proof}[Proof of Theorem \ref{thm:power}]
	Let us fix $\ell\geq 0$ and $\mcoord \in \Mset$, and recall furthermore that the random variables $\left\{a_{\ell + 2s_0, \mcoord + \mathbf{s}}, s_0 \in D_0\left(\ell\right), \mathbf{s}\in Z_{\ell,m}^{\mathbf{Q}}\right\}$ are uncorrelated with variance $C_{\ell+2s_0}$. The variance of $\alalm$ is, thus, given by 
	\begin{align*}
	\Var{\alalm} &= \sum_{s_0 \in D_0\left(\ell\right)} \sum_{\mathbf{s} \in Z_{\ell, \mcoord}^{\mathbf{Q}}} \left( \prod_{j=1}^{d-1} \Anorm^2 h_{m_{j-1}+2s_{j-1},m_{j}+2s_{j};j}^2\right.\\ & \quad \left. \cdot\left(I_{m_{j-1},m_{j}}^{Q_{j-1}}\left(m_{j-1}+2s_{j-1},m_{j}+2s_{j}\right)\right)^2 \right) \\&\quad \cdot 
	\Var{a_{\ell + 2s_0, m_1+2s_1, \ldots, m_{d-1}+2s_{d-1}}} \\
	&= \sum_{s_0 \in D_0\left(\ell\right)} \sum_{\mathbf{s} \in Z_{\ell, \mcoord}^{\mathbf{Q}}} \left( \prod_{j=1}^{d-1} \Anorm^2 h_{m_{j-1}+2s_{j-1},m_{j}+2s_{j};j}^2\right.  \\&\quad \left.\cdot  \left(I_{m_{j-1},m_{j}}^{Q_{j-1}}\left(m_{j-1}+2s_{j-1},m_{j}+2s_{j}\right)\right)^2 \right) C_{\ell+2s_0} \\
	&= \sum_{s_0 \in D_0\left(\ell\right)} V_{\ell, '
		\mcoord} ^Q \left(\ell^\prime\right) C_{\ell+2s_0}.
	\end{align*}
	Using this result in \eqref{eq:spectrualias} completes the proof.
\end{proof}
\begin{proof}[Proof of Theorem \ref{thm:band}]
	First of all, let us consider the harmonic coefficient $\alm$ and study its aliases, denoted by $a_{\ell^\prime, \mcoord^\prime}$, under Condition \ref{cond:uni}, with $Q=Q_0=\ldots=Q_{d-2}>L_0$ and $M>L_0$. For any $\ell^\prime \geq m^\prime_1\geq \ldots \geq m^\prime_{d-2}$, note that
	\begin{equation*} 
	a_{\ell^\prime,\mcoord^\prime}= a_{\ell^\prime,m^\prime_1,\ldots,m^\prime_{d-2},m_{d-1}^\prime} = 0, \quad \text {for any } m_{d-1}^\prime>M>L_0.
	\end{equation*}
	Thus $a_{\ell,m_1,\ldots,m_{d-2},m_{d-1}+2rM} = 0$ for any $r \neq 0$.
	Recalling that 
	\begin{equation*}
	a_{\ell^\prime,m^\prime_1,\ldots,m^\prime_{d-2},m_{d-1}} \quad \text{ for any } m^\prime_{d-2}\geq Q> L_0,
	\end{equation*}
	we obtain that 
	\begin{equation*}
	a_{\ell^\prime,m^\prime_1,\ldots,m_{d-2}+2s_{d-2},m_{d-1}} = 0 \quad \text{ for any } s_{d-2} \geq Q-m_{d-2}. 
	\end{equation*}
	Using now \eqref{eqn:Id} leads to $s_{d-2}=0$. Reiterating this backward procedure for the other harmonic numbers $m_{j}^\prime$, $j=d-3,\ldots,1$ and $\ell^\prime$ yields \eqref{eqn:mainband}.\\
	To prove \eqref{eqn:mainband2}, it suffices to use the band-width in the expansion \eqref{eqn:Fourier_expansion}, that is,
	\begin{equation*}
	T\left(x\right) = \sum_{\ell=0}^{L} \sum_{m\in \Mset }\alalm\Y\left(x\right).
	\end{equation*}
	Using now in the equation above \eqref{eqn:aliaseries}, \eqref{eq:thmmain}, and \eqref{eqn:mainband} yields the claimed result.
\end{proof}
\begin{proof}[Proof of Theorem \ref{thm:bandcov}]
	First, since the power spectrum is band-limited, it holds that $C_{\ell+2s_0} = 0$ for $s_0 \geq (Q-\ell)/2$. Furthermore, for $0\leq \ell \leq Q$ and $\mcoord\in \Mset$, if $s_0 \in \left[-\ell/2,\left(Q-\ell \right)/2-1\right]$,  we obtain that
	\begin{align*}
	s_1 \in \left[-m_1/2,\frac{\ell-m_1}{2}+s_0\right] \subseteq \left[-m_1/2,\frac{Q-m_1}{2}-1\right].
	\end{align*}
	Consequently, simple algebraic manipulations leads to 
	\begin{equation*}
	s_j \in \left[-m_j/2,\frac{\ell-m_j}{2}+s_{j-1}\right] \subseteq \left[-m_j/2,\frac{Q-m_j}{2}-1\right],
	\end{equation*}
	for any $j=1,\ldots,d-2$.\\
	Thus, it follows that, for $s_{d-2}  \in \left[-m_d-2/2,\frac{Q-m_d}{2}-1\right]$ and $Q \geq M > P_L$, $R^M_{m_{d-1}}\left(m_{d-2}+s_{d-2}\right)=\left\{0\right\}$, and, then, $r=0$. Then, by using \eqref{eqn:Id} backward from $j=d-2$ to $j=1$ with any element of the product in \eqref{eq:Ifunc} yields $s_j =0$ for $j=0,\ldots,d-2$. It follows that $V_{\ell, 
		\mcoord} ^\mathbf{Q} \left(\ell^\prime\right)=0$ and $\Var{\alalm} = C_\ell=\Var{\alm}$, as claimed.
\end{proof}
\subsection{Proofs of the auxiliary results}\label{sub:proofaux}
\begin{proof}[Proof of Lemma \ref{lemma:sampling}]
	The symmetry of the sampling angles follow the symmetry of the roots of the Gegenbauer polynomials. Furthermore, note that 
	\begin{equation*}
	\sin \thetacord{j}_{Q_{j-1}-k_{j-1}-1}= \sin \left(\pi -  \thetacord{j}_{k_{j-1}} \right) = \sin \thetacord{j}_{k_{j-1}} .
	\end{equation*}
	Then, we have that
	\begin{align*}
	&\omega_{Q_{j-1}-k_{j-1}-1}^{\left(j\right)} \\ & \qquad= \frac{1}{\int_{-1}^{1} \left(1-t^2\right)^{d-1-j} \diff t}	\int_{-1}^{1} \left(1-t^2\right)^{d-1-j} \lambda_{Q_{j-1}-k_{j-1}-1}\left(t\right) \diff t\\
	 & \qquad=\frac{1}{\int_{-1}^{1} \left(1-t^2\right)^{d-1-j} \diff t}	\int_{-1}^{1} \left(1-t^2\right)^{d-1-j} \!\!\!\! \prod _{i=0, \atop i\neq \left(Q_{j-1}-k_{j-1}-1\right)}^{r-1}\! \! \! \! \frac{t-t_i}{t_i-t_{Q_{j-1}-k_{j-1}-1}} \diff t\\
	 & \qquad=\frac{1}{\int_{-1}^{1} \left(1-t^2\right)^{d-1-j} \diff t}	\int_{-1}^{1} \left(1-t^2\right)^{d-1-j} \prod _{i=0,\atop i\neq \left(k_{j-1}\right)}^{r-1} \frac{t-t_i}{t_i-t_{k_{j-1}}} \diff t\\
	 & \qquad=	\omega_{k_{j-1}}^{\left(j\right)}.
	\end{align*}
	so that $	\wpeso{k_{j-1}}{j} = \wpeso{Q_{j-1}-k_{j-1}-1}{j}$, as claimed.
\end{proof}
\begin{proof}[Proof of Lemma \ref{lemma:simmetry}]
	First of all, note that this result for $d=2$, involving thus Legendre polynomials, has been already claimed in \cite[Theorem 2.1]{LiNorth}. \\
	As far as $d>2$ is concerned, let us preliminarily recall that, for $t\in \left[-1,1\right]$, $\gegenbauergen{n}{\alpha}\left(-t\right)= \bra{-1}^{n} \gegenbauergen{n}{\alpha}\left(t\right)$ (see, for example, \cite[Formula 4.7.4]{szego}). Thus, simple trigonometric identities yield
	\begin{align*}
	G_j\left(\pi - \psi\right) & = \gegenbauer{m_{j-1}}{m_{j}}{j}\left(\cos\left(\pi-\psi\right)\right)\gegenbauer{m^\prime_{j-1}}{m^\prime_{j}}{j}\left(\cos\left(\pi-\psi\right)\right) \sin\left(\pi-\psi\right)^{d-j}\\
	& =\gegenbauer{m_{j-1}}{m_{j}}{j} \left(-\cos\psi\right) \gegenbauer{m^\prime_{j-1}}{m^\prime_{j}}{j}\left(-\cos \psi\right) \left(\sin\psi\right)^{d-j}\\
	& = \bra{-1}^{m_{j-1}+m_{j-1}^\prime-m_{j}-m_{j}^\prime} \gegenbauer{m_{j-1}}{m_{j}}{j}\left(\cos\psi\right)\\ & \, \, \,\,  \cdot \gegenbauer{m^\prime_{j-1}}{m^\prime_{j}}{j}\left(\cos \psi\right) \left(\sin\gegenbauer{m^\prime_{j-1}}{m^\prime_{j}}{j}\right)^{d-j}\\
	& = \bra{-1}^{m_{j-1}+m_{j-1}^\prime-m_{j}-m_{j}^\prime}G_j\left( \psi\right),
	\end{align*}
	as claimed. \\
	In order to prove \eqref{eq:parity3}, consider initially only even values of $Q$. Hence, by means of Lemma \ref{lemma:sampling}, we have that 
	\begin{align*}
	\notag			\sum_{k=0}^{Q-1} w_k G_j\left(\psi_k\right) = & \sum_{k=0}^{\left[Q/2\right]}\left(w_k G_j\left(\psi_k\right) + w_{Q-k-1}G_j\left(\psi_{Q-k-1}\right) \right)\\
	\notag			=& \sum_{k=0}^{\left[Q/2\right]}w_{k}\left(G_j\left(\psi_k\right) + G_j\left(\pi - \psi_{k}\right) \right)\\
	\notag
	=& \sum_{k=0}^{\left[Q/2\right]}w_k\left(G_j\left(\psi_k\right) + \left(-1\right)^{2c+1}G_j\left(\psi_{k}\right) \right)=0.
	\end{align*}
	Moreover, if $Q$ is odd, since sampling points have to be symmetric with respect to $\pi/2$, the  additional point with respect to the previous case has to coincide with $\pi/2$. Thus $G\left(\pi/2\right)=0$ and \eqref{eq:parity3} holds, as claimed.  
\end{proof}
\revise{\begin{proof}[Proof of Corollary \ref{cor:aliasing}] This proof follows directly from the proof of Theorem \ref{thm:main}--Part 2. Indeed, if $M\geq Q$, it follows that $r=0$. Then, combining \eqref{eqn:lago}, \eqref{eqn:lago1} and $\eqref{eqn:Id}$ yields the claimed result.
\end{proof}}
\bibliography{bibliografia}
\end{document}